\documentclass[10.5pt, letter]{amsart}
\usepackage{amsmath}
  \usepackage{paralist}
  \usepackage[usenames]{color}
  \usepackage{graphics} 
  \usepackage{epsfig} 
 \usepackage[colorlinks=true]{hyperref}
\hypersetup{urlcolor=blue, citecolor=red}
\newtheorem{theorem}{Theorem}[section]

\newtheorem{lemma}[theorem]{Lemma}

\theoremstyle{definition}

\newtheorem{remark}{Remark}

\title[]
      {Reconstruct Lam\'{e} parameters of linear isotropic elasticity system}

\subjclass{Primary: 35R30 ; Secondary: 35Q74.}
 \keywords{Inverse Problems, elasticity system, Lam\'e parameters}

\thanks{Department of Mathematics, University of Washington, Seattle, WA 98195, USA. Email: rylai@uw.edu}

\begin{document}
\maketitle

\centerline{\scshape Ru-Yu Lai }
\medskip
{\footnotesize
 } 

\bigskip


\begin{abstract}
Transient Elastography enables detection and characterization of tissue abnormalities. In this paper we assume that the displacements are modeled by linear isotropic elasticity system and the tissue displacement has been obtained by the first step in hybrid methods. Then we reconstruct the Lam\'{e} parameters of this system from knowledge of tissue displacement. We show that for a sufficiently large number of solutions of the elasticity system and for an open set of the well-chosen boundary conditions, $(\lambda,\mu)$ can be uniquely and stably reconstructed. The set of well-chosen boundary conditions is characterized in terms of appropriate complex geometrical optics solutions.
\end{abstract}

\section{Introduction}

Medical imaging is the technique and process used to create images of the human body for clinical purposes or medical science. Each available imagine method has its advantages and disadvantages.
Medical imaging modalities such as Computerized Tomography (CT), Magnetic Resonance Imaging (MRI) and Ultrasound Imaging (UI) are examples of modalities providing high resolution. In some situations, these modalities fail to exhibit a sufficient contrast between different types of tissues. For instance, in breast imaging ultrasound provides high resolution, while suffers from a low contrast. Other modalities, based on optical, elastic, or electrical properties of these tissues, display high contrast, such as Optical Tomography (OT) and Electrical Impedance Tomography (EIT).
For example, some breast tumors on early stages might have no contrast with the healthy tissues with respect to ultrasound propagation, but a huge contrast in their optical and electric properties.

In order to obtain better image, the natural idea is to try to combine different imaging modalities. These are coupled-physics imaging methods, also called hybrid methods. It is to combine the high resolution modality with another high contrast modality.
Examples of possible physical couplings include: Photo-Acoustic Tomography (PAT) and Thermo-Acoustic Tomography (TAT), Ultrasound modulated Optical Tomography (UMOT) and elasticity with ultrasound in Transient Elastography (TE).

Reconstructions in hybrid inverse problems involve two steps. The first step is to solve the inverse problem concerning the high-resolution-low-contrast modality. For instance, in PAT and TAT, this corresponds to the reconstructing the initial condition of a wave equation from boundary measures. In Transient Elastography, this is to solving an inverse scattering problem in a time-dependent wave equation. In this paper, we assume that this first step has been performed. We will focus on the second step.
In the second step, we assume that the first step has been done and we try to reconstruct the coefficients that display high contrast from the mappings obtained during the first step.

In this paper, the modality we consider is Transient Elastography.
TE is a non-invasive tool for measuring liver stiffness. The device creases high-resolution shear
stiffness images of human tissue for diagnostic purposes. Shear
stiffness is targeted because shear wave speed is larger in abnormal
tissue than in normal tissue. In the experiment tissue initially is
excited with pulse at the boundary. This pulse creates the shear wave
passing through
the liver tissue.
Then the tissue displacement is measured by using the ultrasound. The displacement is related to the tissue stiffness because the soft tissue has larger deformation than stiff tissue. When we have
tissue displacement, we want to reconstruct shear modulus $\mu$ and the first parameter $\lambda$.
See \cite{MZM} and references there for more details. In TE, the high resolution modality is also ultrasound.
 The tissue displacement data can be obtained by the ultrasound in the first step. The second step is to recover the Lam\'{e} parameters from the knowledge of the tissue displacement.
In the following paper, we will assume the first step has been performed.

Let $\Omega\subset \mathbb{R}^n,\ n=2,3,$ be an open bounded domain with smooth boundary. Let $u$ be the displacement satisfying the linear isotropic elasticity system
\begin{align}\label{origin}
\left\{\begin{array}{rl}
     \nabla\cdot (\lambda(\nabla\cdot u)I+2S(\nabla u)\mu)+k^2 u=0\ \ &\hbox{in $\Omega$},\\
     u=g\ \ &\hbox{on $\partial\Omega$},
     \end{array}
     \right.
\end{align}
where $S(A)=(A+A^T)/2$ is the symmetric part of the matrix $A$. Here $(\lambda,\mu)$ are Lam\'{e} parameters and $k\in \mathbb{R}$ is the frequency. Assume that $k$ is not the eigenvalue of the elasticity system.
The set of internal functions obtained by ultrasound in TE is given by :
\begin{align*}
    H(x)=(u_j(x))_{1\leq j\leq J}\ \ \ \hbox{in $\Omega$}
\end{align*}
for some positive integer $J$.

Denote that $\mathcal{P}=\{(\lambda,\mu)\in C^{7}(\overline{\Omega})\times C^9(\overline{\Omega});\ 0<m\leq \|\lambda\|_{C^{7}(\overline{\Omega})}, \|\mu\|_{C^{9}(\overline{\Omega})}\leq M\ \  \hbox{and}\ \   \lambda,\ \mu>0 \}$.
Let $H$ and $\tilde H$ be two sets of internal data of the elasticity system with parameters $(\lambda, \mu)$ and $(\tilde \lambda, \tilde\mu)$, respectively. Below is our main result:
\begin{theorem}

Let $\Omega$ be an open bounded domain of $\mathbb{R}^n$ with smooth boundary. Suppose that the Lam\'{e} parameters $(\lambda, \mu)$ and $(\tilde\lambda,\tilde\mu)\in \mathcal{P}$ and $\mu|_{\partial\Omega}=\tilde\mu|_{\partial\Omega}$.
Let $u^{(j)}$ and $\tilde u^{(j)}$ be the solutions of the elasticity system with boundary data $g^{(j)}$ for parameters $(\lambda, \mu)$ and $(\tilde\lambda,\tilde\mu)$, respectively.
Let $H=(u^{(j)})_{1\leq j\leq J}$ and $\tilde H=(\tilde u^{(j)})_{1\leq j\leq J}$ be the corresponding internal data for $(\lambda, \mu)$ and $(\tilde\lambda,\tilde\mu)$, respectively for some integer $J\geq 3n+1$ .

Then there is an open set of the boundary data $(g^{(j)})_{1\leq j\leq J}$ such that
if $H=\tilde H $
implies $(\lambda,\mu)=(\tilde \lambda,\tilde\mu)$ in $\Omega$.\\
Moreover, we have the stability estimate
$$
     \|\mu-\tilde\mu\|_{C(\Omega)}+\|\lambda-\tilde\lambda\|_{C(\Omega)}\leq C\|H-\tilde H \|_{C^2(\Omega)}.
$$
\end{theorem}

The remainder of this paper is organized as follows. In section \ref{CGO}, we introduce the CGO solutions of the elasticity system. Section \ref{global2} is devoted to constructing the Lam\'{e} parameters in two-dimensional case. The reconstruction of $(\lambda,\mu)$ in 3D is presented in section \ref{global3}.

\section{Complex geometric optics solutions of the elasticity system}\label{CGO}
In this section, we will briefly introduce the complex geometric
optics (CGO) solutions of the elasticity system. Based on \cite{I},
we can derive the following reduced system. Let $\left(
   w,
   f
\right)^T$ satisfy
\begin{align}\label{wfa}
    \Delta \left(\begin{array}{c}
   w \\
   f
\end{array}
\right)+V_1(x)\left(\begin{array}{c}
   \nabla f \\
   \nabla\cdot w
\end{array}
\right)+V_0(x)\left(\begin{array}{c}
   w \\
   f
\end{array}
\right)=0.
\end{align}
Here $V_0$ contains the third derivative of $\mu$ and
$$
     V_1(x)=\left(
            \begin{array}{cc}
              -2\mu^{1/2}\nabla^2\mu^{-1}+\mu^{-3/2}k^2 & -\nabla \log\mu \\
              0 & \frac{\lambda+\mu}{\lambda+2\mu}\mu^{1/2} \\
            \end{array}
          \right).
$$
Then the solution of the elasticity system (\ref{origin}) is
$$
     u:=\mu^{-1/2}w+\mu^{-1}\nabla f-f\nabla\mu^{-1}.
$$
Here $\nabla^2 g$ denotes the Hessian matrix $\partial^2g/\partial x_i\partial x_j$. Note that we will not need the explicit form of $V_0$ in the construction of CGO solution.

The construction of CGO solutions of (\ref{wfa}) with linear phase was first deduced by Nakamura and Uhlmann in \cite{NU1} and \cite{NU2}, where they introduced the intertwining property in handling the first order term. Eskin and Ralston \cite{ER} also gave a similar result in 2002. Later Uhlmann and Wang \cite{UW} used Carleman estimate to deduce the CGO solutions in two-dimensional case.

In the following sections, we will use the CGO solutions constructed in Eskin and Ralston's paper \cite{ER}. For the rest of this section, we will introduce the key lemma from \cite{ER}.

To construct CGO solutions of (\ref{wfa}), it is convenient to work on $\mathbb{R}^n$ instead of $\Omega$. Since $\Omega$ is bounded, we pick a ball $B_R$ for $R>0$, such that $\overline{\Omega}\subset B_R$ and extend $\lambda$ and $\mu$ to $\mathbb{R}^n$ by preserving its smoothness and also supp$\lambda$, supp$\mu\subset B_R$.

Let $\alpha$ and $\beta$ be two orthogonal unit vectors in $\mathbb{R}^n$.  Denote that $\rho=\tau(\alpha+i\beta)\in \mathbb{C}^n$ and $\theta=\alpha+i\beta$ with $\tau>0$. Eskin and Ralston \cite{ER} showed the following important result in three-dimensional case. For $n=2$, it still holds.
\begin{lemma}(Eskin-Ralston)\label{Eskin}
Consider the Schr$\ddot{o}$dinger equation with external Yang-Mills potentials
\begin{align}\label{Lu}
     Lu=-\Delta u-2iA(x)\cdot \nabla u+B(x)u=0,\ \ \ \ \ x\in B_R\subset\mathbb{R}^n
\end{align}
where $A(x)=(A_1(x), \ldots, A_n(x))\in C^{n_0}(B_R),\ n_0\geq 6$ with $A_j(x)$ and $B(x)$ are $(n+1)\times (n+1)$ matrices.
Then there are solutions of (\ref{Lu}) of the form
\begin{align*}
     u=e^{i\rho\cdot x} (C_0(x,\theta)g(\theta\cdot x)
                        +O(\tau^{-1})
                        )
\end{align*}
where $C_0\in C^{n_0}(B_R)$ is the solution of
\begin{align*}
      i\theta\cdot \frac{\partial}{\partial x}C_0(x,\theta)=\theta \cdot A(x)C_0(x,\theta)
\end{align*}
with $\det C_0$ is never zero, $g(z)$ is an arbitrary polynomial in complex variables $z$, and $O(\tau^{-1})$ is bounded by $C(1/\tau)$ in $H^l(B_R),\ 0\leq l \leq n_0 -2$.
\end{lemma}

By the lemma above, the CGO solutions of (\ref{wfa}) can be written as follows
\begin{align*}
\left(\begin{array}{c}
   w \\
   f
\end{array}\right)
=e^{i\rho\cdot x}\left(\left(
                         \begin{array}{c}
                           r \\
                           s \\
                         \end{array}
                       \right)
+O(\tau^{-1})\right),
\end{align*}
where $(r,\ s)^T$ is $C_0(x,\theta)g(\theta\cdot x)$. We can write $w=e^{i\rho\cdot x}(r+O(\tau^{-1}))$ and  $f=e^{i\rho\cdot x}(s+O(\tau^{-1}))$. Plug $(w,f)^T$ into $ u=\mu^{-1/2}w+\mu^{-1}\nabla f-f\nabla\mu^{-1}$. Then we have the CGO solutions of the elasticity system. Note that $(r, s)^T$ satisfies the equation
\begin{align}\label{rs}
     -2\theta\cdot\nabla \left(
                         \begin{array}{c}
                           r \\
                           s \\
                         \end{array}
                       \right)
     =V_1(x)
      \left(
            \begin{array}{cc}
              0_{n\times n} & \theta \\
              \theta^T & 0 \\
            \end{array}
          \right)
     \left(
                         \begin{array}{c}
                           r \\
                           s \\
                         \end{array}
                       \right).
\end{align}

\begin{remark}
Since $C_0$ is invertible at every point in $\Omega$, we can conclude that $(r,\ s)^T$ is not zero everywhere in $\Omega$. This does not imply that both $r$ and $s$ never vanish in $\Omega$. However, for any point $y\in \Omega$, there is a small neighborhood $B_y$ of $y$ in $\Omega$ and a CGO solution of $(w, f)^T$ such that both $r$ and $s$ do not vanish in $B_y$.
\end{remark}

\section{Reconstruction of Lam\'{e} parameters in two-dimensional case}\label{global2}

In the previous section, we already have the CGO solutions of the elasticity system. Now we want to use them to give a reconstruction of $\mu$ first.
Let $u=(u_1,u_2)^T$ be the displacement which satisfies the elasticity system
\begin{align}\label{uu}
     \nabla\cdot (\lambda(\nabla\cdot u)I+2S(\nabla u)\mu)+k^2 u=0.
\end{align}
We will recover $\mu$ and $\lambda$ separately in the following two subsections.
\subsection{Reconstruction of $\mu$ in 2D}\label{localmu}

From (\ref{uu}), we can deduce the following equation
\begin{align}\label{equ11}
     u^\sharp \cdot F+u^\flat\cdot G=-k^2u^*.
\end{align}
Here we denote
\begin{align}\label{sharp}
     u^\sharp =\left(\begin{array}{c}
         \partial_1(\nabla\cdot u) \\
         \partial_2(\nabla\cdot u) \\
         \nabla\cdot u \\
         \nabla\cdot u \\
       \end{array}
     \right),\ \
     F= \left(
       \begin{array}{c}
         \lambda+\mu \\
         \lambda+\mu \\
         \partial_1 (\lambda+\mu) \\
         \partial_2 (\lambda+\mu) \\
       \end{array}
     \right), \ \
     u^\flat= \left(
       \begin{array}{c}
                a+b \\
                a-b \\
         \partial_1(a+b) \\
         \partial_2(a-b) \\
       \end{array}
     \right),\ \
      G= \left(
       \begin{array}{c}
        \partial_1\mu \\
        \partial_2\mu \\
         \mu \\
         \mu \\
       \end{array}
     \right)
\end{align}
with $a=\partial_2 u_1+\partial_1 u_2,\ b=\partial_1 u_1-\partial_2
u_2$, and $u^*=(u_1+u_2)$. The component $u^\sharp$ and $u^\flat$
are known since they are vectors which only depend on the internal data
$u$. In order to recover $\mu$, we want to eliminate the first term
of (\ref{equ11}) so that $\mu$ satisfies the transport equation.

Obverse that the vector $u^\sharp$ has three different entries, we
only need to construct three linearly independent vectors on some
subdomain in $\Omega$. With these three vectors, we can remove the
first term of the left hand side of the equation (\ref{equ11}). More
precisely, suppose that $u^{(j)}$, for $j=0,1,2$, are three different
solutions of (\ref{equ11}) which satisfy
\begin{align}
       u^{(j)\sharp}\cdot F+ u^{(j)\flat}\cdot G=-k^2u^{(j)*}.
\end{align}

The notations $\Re f$ and $\Im f$ are defined to be the real and
imaginary part of $f$, respectively.
Now for $j=1,2,3$ and $\natural=\sharp,\flat, *$, we let
$$
\mathfrak{u}_1^{(0)\natural}=\Re \chi(x)u^{(0)\natural},\ \ \mathfrak{u}_2^{(0)\natural}=\Im \chi(x)u^{(0)\natural}
$$
and
$$
   \mathfrak{u}^{(1)\natural}=\Re\chi_3(x)u^{(1)\natural},\ \ \mathfrak{u}^{(2)\natural}=\Im\chi_3(x)u^{(1)\natural},\ \
   \mathfrak{u}^{(3)\natural}=\Re\left(\chi_1(x)u^{(1)\natural}+\chi_2(x)u^{(2)\natural}\right),
$$
where $\chi_j(x)$ is a nonzero function. Then we get
\begin{align}\label{equ123}
       \mathfrak{u}_l^{(0)\sharp}\cdot F+ \mathfrak{u}_l^{(0)\flat}\cdot G=-k^2\mathfrak{u}_l^{(0)*},\ \ l=1,2
\end{align}
and
\begin{align}\label{equ13}
       \mathfrak{u}^{(j)\sharp}\cdot F+ \mathfrak{u}^{(j)\flat}\cdot G=-k^2\mathfrak{u}^{(j)*} \ \ \hbox{for $j=1,2,3$.}
\end{align}
Assume that $\left\{\mathfrak{u}^{(1)\sharp},
\mathfrak{u}^{(2)\sharp}, \mathfrak{u}^{(3)\sharp}\right\}$ are three
linearly independent vectors in some subdomain of $\Omega$, say
$\Omega_0$.
Then there exist three functions
$\Theta^l_1, \Theta^l_2$, and $\Theta^l_3$ such that $
\mathfrak{u}_l^{(0)\sharp}+\sum^3_{j=1}\Theta^l_j
\mathfrak{u}^{(j)\sharp}=0$. For $l=1,2$, multiplying (\ref{equ13})
by $\Theta^l_j$ and summing over $j$ with equation (\ref{equ123}), we
have
\begin{align*}
       \mathrm{v}_l\cdot G=-k^2\left( \mathfrak{u}_l^{(0)*}+\sum^3_{j=1}\Theta^l_j  \mathfrak{u}^{(j)*}\right),
\end{align*}
where $\mathrm{v}_l=
\mathfrak{u}_l^{(0)\flat}+\sum^3_{j=1}\Theta^l_j
\mathfrak{u}^{(j)\flat}$. Let $\beta_l=(\mathrm{v}_l \cdot e_1)
\tilde e_1+ (\mathrm{v}_l \cdot e_2)\tilde e_2$ and
$\gamma_l=(e_3+e_4) \cdot \mathrm{v}_l$. Here $e_j\in\mathbb{R}^4,\
\tilde e_j\in \mathbb{R}^2 $ with the $j^{th}$ entry is $1$ and
others are zero. Then the above equation can be rewritten as
$$
    \beta_l\cdot \nabla\mu+\gamma_l \mu=-k^2\left( \mathfrak{u}_l^{(0)*}+\sum^3_{j=1}\Theta^l_j  \mathfrak{u}^{(j)*}\right).
$$
Suppose that $\beta_1(x)$ and $\beta_2(x)$ are linearly independent
for every $x\in \Omega_0$. Then we can recover $\mu$ in
$\Omega_0\subset\Omega$ for each frequency $k$ independently.

\begin{lemma}\label{prelim}
  Let $u^{(j)}$ for $0\leq j\leq 2$ be $C^2$ solutions of the elasticity system with boundary conditions $u^{(j)}=g^{(j)}$ on $\partial\Omega$.
  Let us define
  $
     u^\sharp =\left(
         \partial_1(\nabla\cdot u), \
         \partial_2(\nabla\cdot u), \
         \nabla\cdot u, \
         \nabla\cdot u
     \right)^T
  $
  and assume that
  \begin{enumerate}
   \item $\left\{ \mathfrak{u}^{(1)\sharp},  \mathfrak{u}^{(2)\sharp},  \mathfrak{u}^{(3)\sharp}\right\}$ are three linearly
independent vectors in $\Omega_0$, the neighborhood of $x_0$ in
$\Omega$.
   \item  $\left\{\beta_1(x), \beta_2(x)\right\}$ are linearly independent
   in $\Omega_0$.
  \end{enumerate}
  Then the reconstruction is stable in the sense that
\begin{align}
    \|\mu-\tilde\mu\|_{C(\Omega_0)}\leq C \left(|\mu(x_0^+)-\tilde\mu(x_0^+)|
    +\|H_{x_0}-\tilde H_{x_0}\|_{C^2(\Omega_0)}\right)
\end{align}
where $x_0^+ \in \partial\Omega_0$ and $H_{x_0}=(u^{(j)})_{0\leq j
\leq 2}$.
\end{lemma}

\begin{proof}
  Since $\left\{\beta_1(x), \beta_2(x)\right\}$ are linearly independent
   in $\Omega_0$, we can construct two vector-valued functions
   $\Gamma(x)$, $\Phi(x)\in C(\Omega_0)$ such that
   \begin{align}\label{7V1}
           \nabla \mu+\Gamma(x)\mu=\Phi(x) \ \mbox{in $\Omega_0$.}
   \end{align}
     Since $\mu$ and $\tilde\mu$ are solutions of (\ref{7V1}) with coefficients $(\Gamma, \Phi)$ and $(\tilde\Gamma,\tilde \Phi)$, respectively, we have
     $$
         \nabla(\mu-\tilde\mu)+\Gamma(x)(\mu-\tilde\mu)=\tilde\mu\left(\tilde\Gamma(x)-\Gamma(x)\right)+\left(\Phi(x)-\tilde\Phi(x)\right).
     $$
     Let $x\in \Omega_0$, there exists a integral curve $\psi(t)$ with $\psi(0)=x_0^+$ and
     $\psi(1)=x$.
Thus 
\begin{align}\label{mu33}
    (\mu-\tilde\mu)(\psi(t))= &(\mu-\tilde\mu)(\psi(0))e^{-\int_0^t\Gamma(\psi(s))\cdot\psi'(s)ds}   \notag\\
    &+\int^t_0 \tilde\mu\left(\Gamma(\psi(s))-\tilde\Gamma(\psi(s))\right)+(\Phi-\tilde\Phi)(\psi(s))\cdot \psi'(s)ds,\ \ t\in [0,1].
\end{align}
Since $\Gamma$ and $\Phi$ only depend on $\Omega$ and $ u^{(j)},
j=0,1,2,$ and $u^{(j)}$ are in the class of $C^2(\overline\Omega)$,
the value of $|\Gamma-\tilde\Gamma|$ and $|\Phi-\tilde\Phi|$ are
bounded by the sum of $|\partial^\alpha
u^{(j)}-\partial^\alpha\tilde u^{(j)}|$ for $|\alpha|\leq 2$. There
is a constant $C$ such that
\begin{align*}
    |(\mu-\tilde\mu)(\psi(t))|
    \leq  C |(\mu-\tilde\mu)(\psi(0))|+C\|H_{x_0}-\tilde H_{x_0}\|_{C^{2}(\Omega_0)},\ \ t\in [0,1]
\end{align*}
for $\mu,\ \tilde\mu\in \mathcal{P}$. Thus, for any $x\in \Omega_0$,
the value $ |(\mu-\tilde\mu)(x)|$ is controlled by the internal data
and $\mu$ at the boundary point $x_0\in\partial\Omega_0$.


\end{proof}

\subsubsection{Global reconstruction of $\mu$ in 2D} Let $(\lambda, \mu) \in\mathcal{P}$, by Lemma \ref{Eskin},
it implies that $r$ and $s$ are in $C^7$ and $O(\tau^{-1})\in H^5$.
Now we will show that how we can get three linear independent
vectors of the form $\mathfrak{u}^\sharp$. We plug the CGO solutions
$u_\rho=\mu^{-1/2}w_\rho+\mu^{-1}\nabla f_\rho-f_\rho\nabla
\mu^{-1}$ into $u^{\sharp}$. Then we have the expression
\begin{align}
u^{\sharp}_\rho=e^{i\rho\cdot x}\frac{\sqrt{\mu}}{\lambda+2\mu}\left(\left(
       \begin{array}{c}
         -\rho_1 (\rho\cdot r)+O(|\rho|)\\
         -\rho_2 (\rho\cdot r)+O(|\rho|)\\
         i\rho\cdot r \\
         i\rho\cdot r \\
       \end{array}
     \right)
     +O(1)\right)
\end{align}
by using the following equality which is part of the equation
(\ref{rs})
\begin{align}\label{s}
  -2\rho\cdot \nabla s
      =\frac{\lambda+\mu}{\lambda+2\mu} \mu^{1/2} \rho\cdot r.
\end{align}

Now we fix any point $x_0\in\overline\Omega$ and let $\rho=\tau(1,i)=-i\tilde\rho\in \mathbb{C}^2$ with $\tau>0$.
Since, in Lemma \ref{Eskin}, the matrix solution $C_0(x,\theta)$ is invertible, we can choose a constant vector $g^{(0)}$ such that $C_0(x,\theta)g^{(0)}=(r^{(0)},s^{(0)})^T$
with the conditions
\begin{align*}
s^{(0)}(x_0)=1,\ \ s^{(0)}\neq 0 \ \ \hbox{and}\ \ \rho\cdot r^{(0)}\neq 0
\end{align*}
in a neighborhood of $x_0$ in $B_R$, say $U$.
Then we have the CGO solution of the elasticity system, that is,
\begin{align*}
 u^{(0)}_\rho=\mu^{-1/2}w^{(0)}_\rho+\mu^{-1}\nabla f^{(0)}_\rho-f^{(0)}_\rho\nabla\mu^{-1},
\end{align*}
where
\begin{align*}
\left(\begin{array}{c}
   w^{(0)}_\rho \\
   f^{(0)}_\rho
\end{array}\right)
=e^{i\rho\cdot x}\left(\left(
                         \begin{array}{c}
                           r^{(0)} \\
                           s^{(0)}\\
                         \end{array}
                       \right)
+O(\tau^{-1})\right).
\end{align*}

Let $\theta=\rho/\tau, \tilde\theta=\tilde\rho/\tau$. Let $C_1(x,\theta)$ and $C_2(x,\tilde\theta)$  satisfy the following two equations
$$
        i\theta\cdot \frac{\partial}{\partial x}C_1(x,\theta)=\theta \cdot V_1(x)C_1(x,\theta),\ \ \ i\tilde\theta\cdot \frac{\partial}{\partial x}C_2(x,\tilde\theta)=\tilde\theta \cdot V_1(x)C_2(x,\tilde\theta),
$$
respectively.
Since $\tilde\rho=i\rho$, we can choose $C_2(x,\tilde\theta)=C_1(x,\theta)$. Moreover,
with suitable constant vector $g$, we can get that $r^{(2)}=r^{(1)}$ and $s^{(2)}=s^{(1)}$ and
\begin{align*}
 s^{(l)}(x_0)=0\ ,\ \
 r^{(1)}(x_0)=(1,-i)=r^{(2)}(x_0).
\end{align*}
By continuity of $r^{(l)}$, we have $\rho\cdot r^{(l)}\neq 0$ in a neighborhood $U$ of $x_0$.
Then the CGO solutions are
\begin{align*}
 u_\rho^{(l)}=\mu^{-1/2}w_\rho^{(l)}+\mu^{-1}\nabla f_\rho^{(l)}-f_\rho^{(l)}\nabla\mu^{-1},
\end{align*}
where
\begin{align*}
\left(\begin{array}{c}
   w_\rho^{(1)} \\
   f_\rho^{(1)}
\end{array}\right)
=e^{i\rho\cdot x}\left(\left(
                         \begin{array}{c}
                           r^{(1)} \\
                           s^{(1)}\\
                         \end{array}
                       \right)
+O(\tau^{-1})\right),\ \left(\begin{array}{c}
   w_\rho^{(2)} \\
   f_\rho^{(2)}
\end{array}\right)
=e^{i\tilde\rho\cdot x}\left(\left(
                         \begin{array}{c}
                           r^{(2)} \\
                           s^{(2)}\\
                         \end{array}
                       \right)
+O(\tau^{-1})\right).
\end{align*}
So far we have three CGO solutions, that is, $u_\rho^{(0)},\ u_\rho^{(1)}$, and $u_\rho^{(2)}$.

We consider
\begin{align*}
    e^{-i\rho\cdot x}|\rho|^{-2}u^{(1)\sharp}_\rho=& \frac{1}{2}\left(\begin{array}{c}
         -\frac{\sqrt{\mu}}{\lambda+2\mu}(r^{(1)}_1+ir^{(1)}_2) \\
         -\frac{\sqrt{\mu}}{\lambda+2\mu}(ir^{(1)}_1-r^{(1)}_2)\\
         0 \\
         0 \\
       \end{array}
       \right)
     +O(\tau^{-1}),
\end{align*}
\begin{align*}
    e^{-i\rho\cdot x}|\rho|^{-1}u^{(1)\sharp}_\rho=& \frac{1}{\sqrt{2}}\left(\begin{array}{c}
         -\frac{\sqrt{\mu}}{\lambda+2\mu}\tau(r^{(1)}_1+ir^{(1)}_2)+O(1)\\
         -\frac{\sqrt{\mu}}{\lambda+2\mu}\tau(ir^{(1)}_1-r^{(1)}_2)+O(1)\\
         i\frac{\sqrt{\mu}}{\lambda+2\mu}( r^{(1)}_1+ir^{(1)}_2) \\
         i\frac{\sqrt{\mu}}{\lambda+2\mu}( r^{(1)}_1+ir^{(1)}_2) \\
       \end{array}
       \right)
     +O(\tau^{-1})
\end{align*}
and
\begin{align*}
   e^{-i\tilde\rho\cdot x}|\tilde\rho|^{-1}u^{(2)\sharp}_\rho=& \frac{1}{\sqrt{2}}\left(\begin{array}{c}
         -\frac{\sqrt{\mu}}{\lambda+2\mu}\tau(-r^{(2)}_1-ir^{(2)}_2)+O(1) \\
         -\frac{\sqrt{\mu}}{\lambda+2\mu}\tau(-ir^{(2)}_1+r^{(2)}_2)+O(1) \\
         i\frac{\sqrt{\mu}}{\lambda+2\mu}( ir^{(2)}_1-r^{(2)}_2) \\
         i\frac{\sqrt{\mu}}{\lambda+2\mu}( ir^{(2)}_1-r^{(2)}_2) \\
       \end{array}
       \right)
     +O(\tau^{-1}).
\end{align*}
Since $r^{(2)}=r^{(1)}$ and $s^{(2)}=s^{(1)}$, we have
\begin{align*}
     e^{-i\rho\cdot x}|\rho|^{-1}u^{(1)\sharp}_\rho+e^{-i\tilde\rho\cdot x}|\tilde\rho|^{-1}u^{(2)\sharp}_\rho= \frac{1}{\sqrt{2}}\left(\begin{array}{c}
         O(1) \\
         O(1) \\
         i\frac{\sqrt{\mu}}{\lambda+2\mu}( (1+i)r^{(1)}_1+(-1+i)r^{(1)}_2) \\
         i\frac{\sqrt{\mu}}{\lambda+2\mu}((1+i)r^{(1)}_1+(-1+i)r^{(1)}_2) \\
       \end{array}
       \right)
     +O(\tau^{-1}).
\end{align*}
We define
\begin{align*}
\mathfrak{u}_{1\rho}^{(0)\natural}= \Re e^{-i\rho\cdot
x}|\rho|^{-2}u^{(0)\natural}_\rho;\ \
\mathfrak{u}_{2\rho}^{(0)\natural}=\Im
e^{-i\rho\cdot x}|\rho|^{-2}u^{(0)\natural}_\rho;\ \ \natural=\sharp,\flat,*
\end{align*}
$$
\mathfrak{u}^{(1)\natural}_\rho= \Re e^{-i\rho\cdot
x}|\rho|^{-2}u^{(1)\natural}_\rho;
\ \
\mathfrak{u}^{(2)\natural}_\rho= \Im e^{-i\rho\cdot
x}|\rho|^{-2}u^{(1)\natural}_\rho;
$$
and
$$
\mathfrak{u}^{(3)\natural}_\rho= \Re \left(e^{-i\rho\cdot
x}|\rho|^{-1}u^{(1)\natural}_\rho+   e^{-i\tilde\rho\cdot
x}|\tilde\rho|^{-1}u^{(2)\natural}_\rho\right).
$$
Then $\left\{ \mathfrak{u}^{(1)\sharp}_\rho, \mathfrak{u}^{(2)\sharp}_\rho,
\mathfrak{u}^{(3)\sharp}_\rho \right\}$ are linearly independent in a small
neighborhood $U$ of $x_0$ when $\tau$ is sufficiently large.

Therefore, for $l=1,2$, there exist functions $\Theta^l_j,\
j=1,2,3$ such that
\begin{align}\label{cond1}
\mathfrak{u}_{l\rho}^{(0)\sharp}+\sum^{3}_{j=1}\Theta^l_j\mathfrak{u}^{(j)\sharp}_\rho =0,\ \ l=1,2.
\end{align}
Since $u_\rho^{(0)},\ u_\rho^{(1)}$, and $u_\rho^{(2)}$ are
solutions of the equation (\ref{equ11}), we have the following
equations:
\begin{align*}
\mathfrak{u}_{l\rho}^{(0)\sharp} \cdot F+\mathfrak{u}_{l\rho}^{(0)\flat}\cdot G=-k^2\mathfrak{u}_{l\rho}^{(0)*},\ \ l=1,2;\\
\mathfrak{u}^{(j)\sharp}_\rho \cdot F+\mathfrak{u}^{(j)\flat}_\rho\cdot G=-k^2\mathfrak{u}^{(j)*}_\rho,\ \ j=1,2,3.
\end{align*}

Summing over the equations above and using (\ref{cond1}), we have
two equations
\begin{align}\label{refB}
     \beta_{\rho,l}\cdot \nabla\mu
     +
     \gamma_{\rho,l} \mu
=-k^2 \left(\mathfrak{u}_{l\rho}^{(0)*}+\sum^3_{j=1}\Theta^l_j\mathfrak{u}^{(j)*}_\rho\right),\ \ \ \ \ \ \hbox{ for $l=1,2$,}
\end{align}
where
\begin{align*}
     \beta_{\rho,l}=(\mathrm{v}_l \cdot e_1) \tilde e_1+ (\mathrm{v}_l \cdot e_2)\tilde e_2
\end{align*}
and
\begin{align*}
     \gamma_{\rho,l}=\mathrm{v}_l \cdot e_3 + \mathrm{v}_l \cdot e_4 ,\\
\end{align*}
Here we define $
\mathrm{v}_l=\mathfrak{u}_{l\rho}^{(0)\flat}+\sum^3_{j=1}\Theta^l_j\mathfrak{u}^{(j)\flat}_\rho.
$


\begin{remark}\label{rk}
By choosing suitable $g^{(0)}$, $|\Theta^l_3(x_0)|$ can be as small as we want. To show that, we choose a new constant vector $\hat{g}^{(0)}$, instead of the original $g^{(0)}$, such that $(\hat{r}^{(0)}, \hat{s}^{(0)})^T=C_0(x,\theta)\hat{g}^{(0)}(\theta\cdot x)$ where $\hat{r}^{(0)}(x_0)=r^{(0)}(x_0)/M,\ M>0,\ \hat{s}^{(0)}(x_0)=s^{(0)}(x_0)$ and $(\hat{r}^{(0)}, \hat{s}^{(0)})^T$ satisfies the original assumption, that is, $\hat{s}^{(0)}(x_0)=1$ and $\hat{s}^{(0)}\neq 0$ and $\rho\cdot \hat{r}^{(0)}\neq 0$ in a neighborhood of $x_0$. Then
\begin{align*}
     \Theta^l_3(x_0)=-P^{3j}\Re\left(\frac{1}{2}\left(\begin{array}{c}
         -\frac{\sqrt{\mu}}{\lambda+2\mu}(1,i)\cdot r^{(0)}(x_0)\\
         -\frac{\sqrt{\mu}}{\lambda+2\mu}i(1,i)\cdot r^{(0)}(x_0)\\
         0 \\
         0 \\
       \end{array}
       \right)  +O(\tau^{-1})   \right)\cdot p_j
\end{align*}
       and
\begin{align*}
        \hat\Theta^l_3(x_0)=-P^{3j} \Re\left(\frac{1}{2M}\left(\begin{array}{c}
         -\frac{\sqrt{\mu}}{\lambda+2\mu}(1,i)\cdot r^{(0)}(x_0)\\
         -\frac{\sqrt{\mu}}{\lambda+2\mu}i(1,i)\cdot r^{(0)}(x_0)\\
         0 \\
         0 \\
       \end{array}
       \right)  +O(\tau^{-1})   \right)\cdot p_j,
\end{align*}
where $P=(p_{ij})$ and $P^{-1}=(P^{ij})$ with $p_{ij}= p_i\cdot p_j$. Note that $P$ is a boundedly invertible symmetric matrix. Here $p_1=\mathfrak{u}^{(1)\sharp}_\rho,\ p_2=\mathfrak{u}^{(2)\sharp}_\rho,\ p_3=\mathfrak{u}^{(3)\sharp}_\rho$.
From above, we obtain that $\hat\Theta^l_3$ is close to $\Theta^l_3/M$ as $\tau$ is large.
Therefore, the new $|\hat{\Theta}^l_3(x_0)|$ is small when $M$ and $\tau$ is sufficiently large.
\end{remark}

\begin{lemma}\label{beta11}
Given any point $x_0\in \overline\Omega$, there exists a small neighborhood $V$ of $x_0$ such that $\beta_{\rho,j}$ is not zero in $V$ for $j=1,2$.
\end{lemma}
Note that we denote by $\sim$ equalities up to terms which are asymptotically negligible as $\tau$ goes to infinity.
\begin{proof}
     It is sufficiently to prove that $\beta_{\rho,1}$ does not vanish in some neighborhood of $x_0$. Recall that
\begin{align}\label{betarho}
    \beta_{\rho,1}=(\mathrm{v}_1 \cdot e_1) \tilde e_1+ (\mathrm{v}_1 \cdot e_2)\tilde e_2,
\end{align}
where
\begin{align*}
     \mathrm{v}_1=\mathfrak{u}_{1\rho}^{(0)\flat}+\sum^3_{j=1}\Theta^1_j\mathfrak{u}^{(j)\flat}_\rho.
\end{align*}
By (\ref{cond1}), we have that $\Theta^1_j\sim \tau^0$ for $j=1,2,3$.
Since $s^{(1)}(x_0)=0$, we get that
\begin{align}\label{g}
   \left(\left(\mathfrak{u}^{(1)\flat}_\rho\cdot e_1\right) \tilde e_1+\left(\mathfrak{u}^{(1)\flat}_\rho\cdot e_2\right) \tilde e_2\right)(x_0)   \sim \tau^{-1}
\end{align}
and
\begin{align}\label{gg}
   \left(\left(\mathfrak{u}^{(2)\flat}_\rho\cdot e_1\right) \tilde e_1+\left(\mathfrak{u}^{(2)\flat}_\rho\cdot e_2\right) \tilde e_2\right)(x_0)   \sim \tau^{-1}.
\end{align}
Since $s^{(k)}(x_0)=0$, $r^{(2)}=r^{(1)}$, $s^{(2)}=s^{(1)}$, and $(1,-i)\cdot r^{(1)}(x_0)=0$, we obtain that
\begin{align}\label{ggg}
     &\left(      \left(\mathfrak{u}^{(3)\flat}_\rho\cdot e_1\right)\tilde e_1 +\left(\mathfrak{u}^{(3)\flat}_\rho\cdot e_2\right)\tilde e_2   \right)(x_0)   \notag\\
     =&\ 4\mu^{-1}\Re \frac{1}{\sqrt{2}}    \left(\left(
       \begin{array}{c}
        i\partial_1 s^{(1)}+\partial_2 s^{(1)} \\
        -\partial_1 s^{(1)}+i\partial_2 s^{(1)} \\
       \end{array}
     \right)
     +O(\tau^{-1})\right)(x_0).
\end{align}

Combining from (\ref{betarho}) to (\ref{ggg}), since $\Theta^1_3(x_0)$ can be taken as small as we want (See Remark 2) in the construction of CGO solutions above, it follows that
\begin{align*}
\beta_{\rho,1}(x_0)\sim \mu^{-1} \left(\left(
       \begin{array}{c}
         2 \\
         -2 \\
       \end{array}
     \right)
     +
      \Theta^1_3\Re
      \left(\begin{array}{c}
         \frac{4}{\sqrt{2}}(i\partial_1 s^{(1)}+\partial_2 s^{(1)}) \\
         \frac{4}{\sqrt{2}}(-\partial_1 s^{(1)}+i\partial_2 s^{(1)}) \\
       \end{array}
     \right)\right)(x_0)
     \neq 0.
\end{align*}
Similarly,
\begin{align*}
\beta_{\rho,2}(x_0)\sim \mu^{-1} \left(\left(
       \begin{array}{c}
         2 \\
         2 \\
       \end{array}
     \right)
     +
      \Theta^2_3\Re
      \left(\begin{array}{c}
         \frac{4}{\sqrt{2}}(i\partial_1 s^{(1)}+\partial_2 s^{(1)}) \\
         \frac{4}{\sqrt{2}}(-\partial_1 s^{(1)}+i\partial_2 s^{(1)}) \\
       \end{array}
     \right)\right)(x_0)
     \neq 0.
\end{align*}
By continuity of $\beta_{\rho,j}$, we complete the proof.
\end{proof}

Let $g^{(j)}_\rho=u_\rho^{(j)}|_{\partial\Omega}$ be the given boundary data. By Lemma \ref{Eskin}, since $(\lambda,\mu)\in \mathcal{P}$, we knew that $u_\rho ^{(j)}\in H^4(B_R)$. Let $g^{(j)}\in C^{1,\alpha}(\partial\Omega)$ be the boundary data close to $g^{(j)}_\rho $, that is,
\begin{align*}
\|g^{(j)}-g^{(j)}_\rho\|_{C^{1,\alpha}(\partial\Omega)}< \varepsilon\ \ \ \hbox{for some $\varepsilon>0$},
\end{align*}
then we can find solutions $u^{(j)}$ of the elasticity system with boundary data $g^{(j)}$(the existence of such solutions, see e.g. Ch.4 of \cite{M}). By elliptic regularity theorem, we have
\begin{align}\label{uuu}
    \|u^{(j)}-u^{(j)}_\rho\|_{C^{2}(\overline\Omega)}
      <C \varepsilon
\end{align}
for some constant $C$ which is independent of $\lambda, \mu$. Then we obtain that
\begin{align*}
     \|\mathfrak{u}^{(j)\sharp}-\mathfrak{u}^{(j)\sharp}_\rho\|_{C^0(\overline\Omega)}< C \varepsilon.
\end{align*}

Here the notation $\mathfrak{u}^{(j)\sharp}$ is constructed in the same way as $\mathfrak{u}^{(j)\sharp}_\rho$ with the CGO solutions $u^{(j)}_\rho$ replaced by the solutions $u^{(j)}.$
Therefore, $\left\{\mathfrak{u}^{(1)\sharp},\mathfrak{u}^{(2)\sharp},\mathfrak{u}^{(3)\sharp}\right\}$ are also linearly independent when $\varepsilon$ is sufficiently small.

We construct $\beta_j$ as in the equation (\ref{refB}) with $u^{(j)}_\rho$ replaced by $u^{(j)}, j=0,1,2$.
Therefore, by the definitions of $\beta_j $ and $\beta_{\rho,j}$ and (\ref{uuu}), it follows that
$$
\|\beta_j-\beta_{\rho,j}\|_{C^1(\overline\Omega)}
$$
is small when $\varepsilon$ is small.
Since $\beta_{\rho,j}$ is not zero in $V$ by Lemma \ref{beta11}, we can deduce that $\beta_j$ is also not zero in $V$ if $\varepsilon$ is small enough and $\tau$ is sufficiently large.
Moreover, with the suitable chosen CGO solutions $u_\rho$,  $\{\beta_{\rho,1}, \beta_{\rho,2}\}$ are linearly independent in $V$ as $\tau$ is sufficiently large (See the proof of Lemma \ref{beta11}). When $\varepsilon$ is sufficiently small, it implies that $\{\beta_1,\beta_2\}$ are also linearly independent in $V$. Then we have the following equations:
\begin{align}\label{muj}
     \beta_l\cdot\nabla\mu+\gamma_l\mu=-k^2 \left(\mathfrak{u}_{l}^{(0)*}+\sum^3_{j=1}\Theta^l_j\mathfrak{u}^{(j)*}\right),\ \  l=1,2
\end{align}
with $\{\beta_1,\beta_2\}$ a basis in $\mathbb{R}^2$ for every point $x\in \Omega_0$. Here we denote $\Omega_0=U\cap V \cap \Omega$.

Thus, there exists an invertible matrix $A=(a_{ij})$ such that $\beta_l=\sum a_{lk}\tilde e_k$ with inverse of class $C^0(\Omega)$. Thus, we have constructed two vector-valued functions $\Gamma(x),\ \Phi(x)\in C(\Omega)$ such that (\ref{muj}) can be rewritten as
\begin{align}\label{V1}
     \nabla\mu+\Gamma(x)\mu=\Phi(x)\ \ \ \ \ \hbox{in\ $\Omega_0$}.
\end{align}
Then we obtain the following result:
\begin{theorem}\label{mulocal}
Suppose that $(\lambda, \mu)$ and $(\tilde\lambda,\tilde\mu)\in \mathcal{P}$. For any fixed $x_0\in \partial\Omega$, let $u^{(j)}_\rho$ be the corresponding CGO solutions for $(\lambda,\mu)$ and $u^{(j)}$ constructed as above with internal data $H_{x_0}=(u^{(j)})_{0\leq j\leq 2}$ and with $\varepsilon$ sufficiently small. Let $\tilde H_{x_0}$ be constructed similarly with the parameters $(\tilde\lambda,\tilde\mu)$. Assume that $\mu|_{\partial\Omega}=\tilde\mu|_{\partial\Omega}$.

Then $H_{x_0}=\tilde H_{x_0}$ implies that $ \mu=\tilde\mu$ in $\Omega_0$, the neighborhood of $x_0$ in $\Omega$.
\end{theorem}

\begin{proof}
Based on the construction above, the domain $\Omega_0$ can be taken as a small open ball with center $x_0$ and $\Omega_0\subset B_R$. Since $H_{x_0}=\tilde H_{x_0}$, we have that $\mu $ and $\tilde\mu$ solve the same equation
$\nabla\mu+\Gamma(x)\mu=\Phi(x)$ in $\Omega_0$ where the functions $\Gamma$ and $\Phi$ depend on $u^{(j)}$.
Let $x\in \Omega_0$ and denote $\psi(t)=(1-t)x_0+tx,\ t\in[0,1]$. Restricted to this curve, we have
\begin{align}\label{solmu}
    \left\{
       \begin{array}{ll}
        \psi'(t)\cdot\nabla \mu+\Gamma(\psi(t))\cdot\psi'(t)\mu=\Phi(\psi(t))\cdot\psi'(t) \ \ \hbox{in $\Omega_0$}\\
        \mu(x_0)=\tilde \mu(x_0),
       \end{array}
     \right.
\end{align}
The solution of (\ref{solmu}) is given by
\begin{align*}
    \mu(\psi(t))=\mu(\psi(0))e^{-\int_0^t\Gamma(\psi(s))\cdot\psi'(s)ds}+\int^t_0 \Phi(\psi(s))\cdot \psi'(s)ds,\ \ \ \ \ t\in [0,1].
\end{align*}
The solution $\tilde\mu(x)$ is given by the same formula since $\mu|_{\partial\Omega}=\tilde\mu|_{\partial\Omega}$ so that $\mu=\tilde \mu$ in $\Omega_0$.
\end{proof}

We have constructed $\left\{\mathfrak{u}^{(1)\sharp},\mathfrak{u}^{(2)\sharp},\mathfrak{u}^{(3)\sharp}\right\}$ are linearly independent and $\{\beta_1,\beta_2\}$ forms a basis in $\mathbb{R}^2$ for every point $x\in \Omega_0$ when $\varepsilon$ is sufficiently small and $\tau$ is large. Applying Lemma \ref{prelim}, we have
\begin{theorem}\label{mu}
Suppose that $(\lambda, \mu)$ and $(\tilde\lambda,\tilde\mu)\in \mathcal{P}$. For any fixed $x_0\in \overline\Omega$, let $u^{(j)}_\rho$ be the corresponding CGO solutions for $(\lambda,\mu)$ and $u^{(j)}$ constructed as above with internal data $H_{x_0}=(u^{(j)})_{0\leq j\leq 2}$ and with $\varepsilon$ sufficiently small. Let $\tilde H_{x_0}=(\tilde u^{(j)})_{0\leq j\leq 2}$ be constructed similarly for the parameters $(\tilde\lambda,\tilde\mu)$ with $u^{(j)}|_{\partial\Omega}=\tilde u^{(j)}|_{\partial\Omega}$.
Then there exists an open neighborhood $\Omega_0$ of $x_0$ in $\Omega$ such that
\begin{align}\label{muu}
     \|\mu-\tilde \mu\|_{C(\Omega_0)}\leq C \left(|\mu(x_0)-\tilde \mu(x_0)|+\|H_{x_0}-\tilde H_{x_0}\|_{C^{2}(\Omega_0)}\right),\ \ \ x_0\in \partial\Omega.
\end{align}
and
\begin{align}\label{muuu}
     \|\mu-\tilde \mu\|_{C(\Omega_0)}\leq C \left(|\mu(x^+_0)-\tilde \mu(x^+_0)|+\|H_{x_0}-\tilde H_{x_0}\|_{C^{2}(\Omega_0)}\right),\ \ \ x^+_0\in \partial\Omega_0,\ x_0\in \Omega.
\end{align}

\end{theorem}




The global uniqueness and stability result are stated as follows.

\begin{theorem}(Global reconstruction of $\mu$)\label{mu2}
Let $\Omega$ be an open bounded domain of $\mathbb{R}^2$ with smooth boundary. Suppose that the Lam\'{e} parameters $(\lambda, \mu)$ and $(\tilde\lambda,\tilde\mu)\in \mathcal{P}$ and $\mu|_{\partial\Omega}=\tilde\mu|_{\partial\Omega}$.
Let $u^{(j)}$ and $\tilde u^{(j)}$ be the solutions of the elasticity system with boundary data $g^{(j)}$ for parameters $(\lambda, \mu)$ and $(\tilde\lambda,\tilde\mu)$, respectively.
Let $H=(u^{(j)})_{1\leq j\leq J}$ and $\tilde H=(\tilde u^{(j)})_{1\leq j\leq J}$ be the corresponding internal data for $(\lambda, \mu)$ and $(\tilde\lambda,\tilde\mu)$, respectively for some integer $J\geq 3$ .

Then there is an open set of the boundary data $(g^{(j)})_{1\leq j\leq J}$ such that
if $H=\tilde H $
implies $\mu=\tilde \mu$ in $\Omega$.\\
Moreover, we have the stability estimate
$$
     \|\mu-\tilde\mu\|_{C(\Omega)}\leq C\|H-\tilde H \|_{C^2(\Omega)}.
$$
\end{theorem}

  Note that for the uniqueness of $\mu$, we suppose that the two set of internal data are the same, that is, $H=\tilde H$. Since $\mu$ is uniquely reconstructed near a fixed point $x_0\in \partial\Omega$ under the condition $\mu|_{\partial\Omega}=\tilde\mu|_{\partial\Omega}$, from the stability of $\mu$ in $\Omega$, we can obtain that $\mu=\tilde\mu$ in $\Omega$.
\begin{proof}
      In section \ref{CGO}, we constructed CGO solutions in a ball $B_R$ which contains $\overline\Omega$. First, we consider any point $x$ in $\partial\Omega$. Then we can find an open neighborhood $B_x\subset B_R$ of $x$. By Theorem \ref{mu}, we have the estimate
   \begin{align}\label{a1}
     \|\mu-\tilde \mu\|_{C(B_x\cap \Omega)}\leq C \|H_x-\tilde H_x\|_{C^{2}(\Omega)}
   \end{align}
   since $\mu|_{\partial\Omega}=\tilde\mu|_{\partial\Omega}$.

   Second, for any point $y\in \Omega$, by Theorem \ref{mu},
   there exists an open neighborhood $B_y\subset\Omega$ of $y$ with $\overline B_y\cap \partial\Omega=\emptyset$ such that
   \begin{align}\label{a2}
     \|\mu-\tilde \mu\|_{C(B_y)}\leq C \left(|\mu(y^+)-\tilde \mu(y^+)|+\|H_y-\tilde H_y\|_{C^{2}(\Omega)}\right)\ \ \hbox{for some}\ y^+\in \partial B_y.
   \end{align}
   Therefore, the compact set $\overline\Omega$ is covered by $\bigcup_{x\in\overline\Omega} B_x$. Then there exists finitely many $B_x$, say, $B_{x_1},\ldots, B_{x_N}$, such that $\overline\Omega\subset \bigcup_{l=1}^N B_{x_l}$.

   Now for arbitrary point $z\in \Omega$, there is $B_{x_j}$ such that $z\in B_{x_j}$. Suppose that $\overline B_{x_j}\cap\partial\Omega\neq \emptyset$, this means that $x_j\in \partial\Omega$. Then, by (\ref{a1}), we have
    \begin{align}\label{b1}
     |\mu(z)-\tilde \mu(z)|   \leq C \|H_{x_j}-\tilde H_{x_j}\|_{C^{2}(\Omega)}.
   \end{align}

    Otherwise, if $\overline{B}_{x_j}\cap \partial\Omega$ is empty, then, by (\ref{a2}), we get that
   \begin{align}\label{b2}
     |\mu(z)-\tilde \mu(z)|  \leq C \left(|\mu(x_j^+)-\tilde \mu(x_j^+)|+\|H_{x_j}-\tilde H_{x_j}\|_{C^{2}(\Omega)}\right) \ \ \hbox{for $x_j^+\in \partial B_{x_j}$}.
   \end{align}
   For the point $x^+_j$, since $\overline \Omega$ is covered by finitely many subdomain $B_{x_l}$, after at most $N-1$ steps, we have
   \begin{align}\label{b3}
     |\mu(x_j^+)-\tilde \mu(x_j^+)|   \leq C   \sum_{l\neq j, l=1}^{N}  \|H_{x_l}-\tilde H_{x_l}\|_{C^{2}(\Omega)}.
   \end{align}
   Combining (\ref{b2}) and (\ref{b3}). Then we have
   \begin{align}\label{29}
     |\mu(z)-\tilde \mu(z)|   \leq C   \sum_{l=1}^N  \|H_{x_l}-\tilde H_{x_l}\|_{C^{2}(\Omega)}.
   \end{align}
   With (\ref{b1}) and (\ref{29}), we have the global stability
   \begin{align}
     \|\mu-\tilde \mu\|_{C(\Omega)}   \leq C   \sum_{l=1}^N  \|H_{x_l}-\tilde H_{x_l}\|_{C^{2}(\Omega)}.
   \end{align}

\end{proof}

\subsection{Reconstruction of $\lambda$ in 2D}\label{reconslambda}
The elasticity system can also be written in this form
\begin{align}\label{equ}
    u^\sharp\cdot F+u^\flat\cdot G
     =-k^2u^*,
\end{align}
where
 $$
     u^\sharp=\left(
       \begin{array}{c}
         \nabla\cdot u \\
         \nabla\cdot u \\
         a+b \\
         a-b \\
       \end{array}
     \right),\ \
     F= \left(
       \begin{array}{c}
         \partial_1 (\lambda+\mu) \\
         \partial_2 (\lambda+\mu) \\
         \partial_1\mu \\
         \partial_2\mu \\
       \end{array}
     \right),\ \
     u^\flat=\left(
       \begin{array}{c}
         \partial_1(\nabla\cdot u) \\
         \partial_2(\nabla\cdot u) \\
         \partial_1(a+b) \\
         \partial_2(a-b) \\
       \end{array}
     \right)
,\ \
       G=\left(
       \begin{array}{c}
         \lambda+\mu \\
         \lambda+\mu \\
         \mu \\
         \mu \\
       \end{array}
     \right).
$$
As in the reconstruction of $\mu$, we will construct there linear independent vectors such that the first term of the equation (\ref{equ}) can be eliminated.

Suppose that $u^{(j)}$, for $j=0,1,2,3$, are three different
solutions of (\ref{equ11}) which satisfy
\begin{align}
       u^{(j)\sharp}\cdot F+ u^{(j)\flat}\cdot G=-k^2u^{(j)*}.
\end{align}
Now for $j=1,2,3$ and $\natural=\sharp,\flat, *$, we let
$$
\mathfrak{u}^{(0)\natural}=\Re \chi(x)u^{(0)\natural},
$$
and
$$
   \mathfrak{u}^{(1)\natural}=\Re\chi_3(x)u^{(1)\natural},\ \ \mathfrak{u}^{(2)\natural}=\Im\chi_3(x)u^{(1)\natural},\ \
   \mathfrak{u}^{(3)\natural}=\Re\left(\chi_1(x)u^{(2)\natural}+\chi_2(x)u^{(3)\natural}\right),\ \
$$
where $\chi_j(x)$ is a nonzero function. Then we get
\begin{align}\label{equ12}
       \mathfrak{u}^{(0)\sharp}\cdot F+ \mathfrak{u}^{(0)\flat}\cdot G=-k^2\mathfrak{u}^{(0)*}
\end{align}
and
\begin{align}\label{equ1}
       \mathfrak{u}^{(j)\sharp}\cdot F+ \mathfrak{u}^{(j)\flat}\cdot G=-k^2\mathfrak{u}^{(j)*} \ \ \hbox{for $j=1,2,3$.}
\end{align}
Assume that $\left\{\mathfrak{u}^{(1)\sharp},
\mathfrak{u}^{(2)\sharp}, \mathfrak{u}^{(3)\sharp}\right\}$ are three
linearly independent vectors in some subdomain of $\Omega$, say
$\Omega_0$.
Then there exist three functions
$\Theta_1, \Theta_2$, and $\Theta_3$ such that $
\mathfrak{u}^{(0)\sharp}+\sum^3_{j=1}\Theta_j
\mathfrak{u}^{(j)\sharp}=0$. Multiplying (\ref{equ1})
by $\Theta_j$ and summing over $j$ with equation (\ref{equ12}), we
have
\begin{align*}
       \mathrm{v}\cdot G=-k^2\left( \mathfrak{u}^{(0)*}+\sum^3_{j=1}\Theta_j  \mathfrak{u}^{(j)*}\right),
\end{align*}
where $\mathrm{v}=
\mathfrak{u}^{(0)\flat}+\sum^3_{j=1}\Theta_j
\mathfrak{u}^{(j)\flat}$. Let
$$
    \kappa=(1,1,0,0)^T\cdot \mathrm{v},\ \ \sigma=-(1,1,1,1)^T\cdot \mathrm{v}.
$$
Then the above equation can be rewritten as
$$
    \kappa\lambda=\sigma\mu-k^2\left( \mathfrak{u}^{(0)*}+\sum^3_{j=1}\Theta_j  \mathfrak{u}^{(j)*}\right).
$$
Suppose that $\kappa(x)$ does not vanish in $\Omega_0$ . Then we can deduce the following lemma.

\begin{lemma}\label{prelim2}
  Let $u^{(j)}$ for $0\leq j\leq 3$ be $C^2$ solutions of the elasticity system with boundary conditions $u^{(j)}=g^{(j)}$ on $\partial\Omega$. Let $u=(u_1,u_2)$, $a=\partial_2 u_1+\partial_1 u_2$ and  $b=\partial_1 u_1-\partial_2 u_2$.
  We define
  $
      u^\sharp=\left(
         \nabla\cdot u,\
         \nabla\cdot u,\
         a+b,\
         a-b
     \right)
  $
  and assume that
  \begin{enumerate}
   \item $\left\{ \mathfrak{u}^{(1)\sharp},  \mathfrak{u}^{(2)\sharp},  \mathfrak{u}^{(3)\sharp}\right\}$ are three linearly
independent vectors in $\Omega_0$, the neighborhood of $x_0$ in
$\Omega$.
   \item  $\kappa(x)$ does not vanish in $\Omega_0$.
  \end{enumerate}
  Then the reconstruction is stable in the sense that
\begin{align}
    \|\lambda-\tilde\lambda\|_{C(\Omega_0)}\leq C \left(\|\mu-\tilde\mu\|_{C(\Omega)}
    +\|H_{x_0}-\tilde H_{x_0}\|_{C^2(\Omega_0)}\right)
\end{align}
where $H_{x_0}=(u^{(j)})_{0\leq j
\leq 3}$.
\end{lemma}

\subsubsection{Global reconstruction of $\lambda$}
We will find three linearly
independent vectors $\left\{ \mathfrak{u}^{(1)\sharp},  \mathfrak{u}^{(2)\sharp},  \mathfrak{u}^{(3)\sharp}\right\}$ first. Then we can deduce stability of $\lambda$ by using Lemma \ref{prelim2} and Theorem \ref{mu2}.

Plugging the CGO solution $ u_\rho=\mu^{-1/2}w_\rho+\mu^{-1}\nabla f_\rho-f_\rho\nabla\mu^{-1}$
into $u^\sharp$,
we get
\begin{align*}
     u^\sharp_\rho=ie^{i\rho\cdot x}\left(\left(
       \begin{array}{c}
         \frac{\sqrt{\mu}}{\lambda+2\mu}\rho\cdot r\\
         \frac{\sqrt{\mu}}{\lambda+2\mu}\rho\cdot r \\
         2\mu^{-1}(is(\rho_1^2+\rho_1\rho_2)+\omega_1\cdot \nabla s)+\mu^{-1/2}\omega_1\cdot r \\
         2\mu^{-1}(is(\rho_1\rho_2-\rho_1^2)+\omega_2\cdot \nabla s)+\mu^{-1/2}\omega_2\cdot r  \\
       \end{array}
     \right)
     +O(1)\right),
\end{align*}
where $\omega_1=(\rho_1+\rho_2, \rho_1-\rho_2)$ and $\omega_2=(\rho_2-\rho_1,\rho_1+\rho_2)$.

For the same fixed point $x_0\in\overline\Omega$. Denote that $\rho=\tau(e_1+ie_2)$ and $\tilde\rho=i\rho$.
We choose a constant vector $g^{(0)}$ such that $C_0(x,\theta)g^{(0)}=(r^{(0)}, s^{(0)})$ with
$$
     s^{(0)}\neq 0,\ \ \rho\cdot r^{(0)}\neq 0,\ \ \rho\cdot r^{(0)}(x_0)=1
$$
in a neighborhood of $x_0$, $U$. Then we get the CGO solution of the elasticity system
$$
  u_\rho^{(0)}=\mu^{-1/2}w_\rho^{(0)}+\mu^{-1}\nabla f_\rho^{(0)}-f_\rho^{(0)}\nabla \mu^{-1}
$$
with
\begin{align*}
\left(\begin{array}{c}
   w_\rho^{(0)} \\
   f_\rho^{(0)}
\end{array}\right)
=e^{i\rho\cdot x}\left(\left(
                         \begin{array}{c}
                           r^{(0)} \\
                           s^{(0)}\\
                         \end{array}
                       \right)
+O(\tau^{-1})\right).
\end{align*}

We choose another constant vector $g^{(1)}$ such that $C_1(x,\theta)g^{(1)}=(r^{(1)}, s^{(1)})$ with
\begin{align*}
s^{(1)}\neq 0  \ \hbox{in $U$ and $\rho\cdot r^{(1)}(x_0)=0$}.
\end{align*}
Then we get the CGO solution of the elasticity system
$$
  u_\rho^{(1)}=\mu^{-1/2}w_\rho^{(1)}+\mu^{-1}\nabla f_\rho^{(1)}-f_\rho^{(1)}\nabla \mu^{-1}
$$
with
\begin{align*}
\left(\begin{array}{c}
   w_\rho^{(1)} \\
   f_\rho^{(1)}
\end{array}\right)
=e^{i\rho\cdot x}\left(\left(
                         \begin{array}{c}
                           r^{(1)} \\
                           s^{(1)}\\
                         \end{array}
                       \right)
+O(\tau^{-1})\right).
\end{align*}

For $l=2,3$, we choose a constant vector $g^{(l)}$ such that $C_l(x,\theta)g^{(l)}=(r^{(l)}, s^{(l)})$ with
$$
\rho\cdot r^{(l)}\neq 0
$$
in $U$. Here we can choose
$$
r^{(2)}=r^{(3)},\ s^{(2)}=s^{(3)}
$$
by taking $g^{(2)}=g^{(3)}$ and $C_2(x,\theta)=C_3(x,\tilde\theta)$. Then we get the CGO solution of the elasticity system
$$
  u_\rho^{(l)}=\mu^{-1/2}w_\rho^{(l)}+\mu^{-1}\nabla f_\rho^{(l)}-f_\rho^{(l)}\nabla \mu^{-1}
$$
with
\begin{align*}
\left(\begin{array}{c}
   w_\rho^{(2)} \\
   f_\rho^{(2)}
\end{array}\right)
=e^{i\rho\cdot x}\left(\left(
                         \begin{array}{c}
                           r^{(2)} \\
                           s^{(2)}\\
                         \end{array}
                       \right)
+O(\tau^{-1})\right),\ \left(\begin{array}{c}
   w_\rho^{(3)} \\
   f_\rho^{(3)}
\end{array}\right)
=e^{i\tilde\rho\cdot x}\left(\left(
                         \begin{array}{c}
                           r^{(3)} \\
                           s^{(3)}\\
                         \end{array}
                       \right)
+O(\tau^{-1})\right).
\end{align*}
We let
$$
   e^{-i\rho\cdot x}i^{-1}|\rho|^{-2}u^{(1)\sharp}_\rho=2i\mu^{-1}s
     \left(
       \begin{array}{c}
         0 \\
         0 \\
         1+i \\
         i-1 \\
       \end{array}
     \right)
     +O(\tau^{-1}),
$$
\begin{align*}
      \mathfrak{u}_\rho^{2,1}:=e^{-i\rho\cdot x}i^{-1}|\rho|^{-1}u^{(2)\sharp}_\rho     =\frac{1}{\sqrt{2}}\left(
       \begin{array}{c}
         \frac{\sqrt{\mu}}{\lambda+2\mu}(1,i)\cdot r^{(2)} \\
         \frac{\sqrt{\mu}}{\lambda+2\mu}(1,i)\cdot r^{(2)}\\
         2\mu^{-1}\tau(1+i)is^{(2)}+\nu^{(2)} \\
         2\mu^{-1}\tau(i-1)is^{(2)}+i\nu^{(2)} \\
       \end{array}
     \right)
     +O(\tau^{-1})
\end{align*}
and
\begin{align*}
      \mathfrak{u}_\rho^{2,2}:=e^{-i\tilde\rho\cdot x}i^{-1}|\tilde\rho|^{-1}u^{(3)\sharp}_\rho
      =\frac{1}{\sqrt{2}}\left(
       \begin{array}{c}
         \frac{\sqrt{\mu}}{\lambda+2\mu}(i,-1)\cdot r^{(3)} \\
         \frac{\sqrt{\mu}}{\lambda+2\mu}(i,-1)\cdot r^{(3)}\\
         2\mu^{-1}\tau(-1-i)is^{(3)}+\nu^{(3)} \\
         2\mu^{-1}\tau(1-i)is^{(3)}+i\nu^{(3)}  \\
       \end{array}
     \right)
     +O(\tau^{-1}),
\end{align*}
where $\nu^{(2)}=2\mu^{-1}(1+i,1-i)\cdot \nabla s^{(2)}+\mu^{-1/2}(1+i,1-i)\cdot r^{(2)}$ and $\nu^{(3)}=2\mu^{-1}(-1+i,i+1)\cdot \nabla s^{(3)}+\mu^{-1/2}(-1+i,i+1)\cdot r^{(3)}$.
To eliminate the higher order term of $\mathfrak{u}_\rho^{2,j}$, we consider the summation of two vectors :
\begin{align*}
    &e^{-i\rho\cdot x}i^{-1}|\rho|^{-1}u^{(2)\sharp}_\rho  +e^{-i\tilde\rho\cdot x}i^{-1}|\tilde\rho|^{-1}u^{(3)\sharp}_\rho\\
    &=\frac{1}{\sqrt{2}}\left(
       \begin{array}{c}
         \frac{\sqrt{\mu}}{\lambda+2\mu}(1+i, i-1)\cdot r^{(3)} \\
         \frac{\sqrt{\mu}}{\lambda+2\mu}(1+i, i-1)\cdot r^{(3)}\\
         2\mu^{-1}(2i, 2)\cdot \nabla s^{(2)}+\mu^{-1/2}(2i,2)\cdot r^{(3)} \\
         2\mu^{-1}(-2,2i)\cdot \nabla s^{(3)}+\mu^{-1/2}(-2,2i)\cdot r^{(3)} \\
       \end{array}
     \right)
     +O(\tau^{-1}).
\end{align*}

We define
$$
     \mathfrak{u}^{(0)\natural}_\rho=\Re e^{-i\rho\cdot x}i^{-1}|\rho|^{-2}u^{(0)\natural}_\rho;\ \natural=\sharp,\flat,*
$$
$$
    \mathfrak{u}^{(1)\natural}_\rho=\Re e^{-i\rho\cdot x}i^{-1}|\rho|^{-2}u^{(1)\natural}_\rho;\ \
    \mathfrak{u}^{(2)\natural}_\rho=\Im e^{-i\rho\cdot x}i^{-1}|\rho|^{-2}u^{(1)\natural}_\rho;
$$
and
$$
     \mathfrak{u}^{(3)\natural}_\rho=\Re\left(e^{-i\rho\cdot x}i^{-1}|\rho|^{-1}u^{(2)\natural}_\rho  +e^{-i\tilde\rho\cdot x}i^{-1}|\tilde\rho|^{-1}u^{(3)\natural}_\rho\right).
$$
Thus we have constructed three linear independent vectors $\left\{\mathfrak{u}^{(1)\sharp}_\rho,\mathfrak{u}^{(2)\sharp}_\rho,\mathfrak{u}^{(3)\sharp}_\rho\right\}$ in $U$ as $\tau$ is sufficiently large. Therefore, there are three functions $\Theta_j,\ j=1,2,3$, such that
\begin{align}\label{sum3.2}
\mathfrak{u}_\rho^{(0)\sharp}+\sum^3_{j=1}\Theta_j\mathfrak{u}_\rho^{(j)\sharp}=0.
\end{align}
They also satisfy the following equations:
\begin{align*}
   \mathfrak{u}_\rho^{(j)\sharp} \cdot F+ \mathfrak{u}_\rho^{(j)\flat}\cdot G= -k^2 \mathfrak{u}_\rho^{(j)*},\ \ 0\leq j\leq 3.
\end{align*}
Summing over $j$ and using (\ref{sum3.2}), we get the following equation
\begin{align}\label{frac}
    \mathrm{v}\cdot G
     =-k^2\left(\mathfrak{u}_\rho^{(0)*}+\sum^3_{j=1}\Theta_j\mathfrak{u}_\rho^{(j)*}\right),
\end{align}
where
\begin{align*}
    \mathrm{v}=\mathfrak{u}_\rho^{(0)\flat}+\sum^3_{j=1}\Theta_j\mathfrak{u}_\rho^{(j)\flat}.
\end{align*}

Then we obtain that
\begin{align}\label{lambdamu}
\kappa\lambda=\sigma\mu-k^2\left(\mathfrak{u}_\rho^{(0)*}+\sum^3_{j=1}\Theta_j\mathfrak{u}_\rho^{(j)*}\right),
\end{align}
where
$$
     \kappa= (1,1,0,0)^T\cdot \mathrm{v},
$$
$$
     \sigma=-(1,1,1,1)^T\cdot \mathrm{v}.
$$
\begin{lemma}\label{kappa}
     $\kappa$ does not vanish in some neighborhood of $x_0$.
\end{lemma}
\begin{proof}
    Since $\Theta_j\sim \tau^0$ and $\rho\cdot r^{(1)}(x_0)=0$, by using similar argument as in the proof of Lemma \ref{beta11}, we have that
    \begin{align}\label{alpha}
         (1,1,0,0)^T\cdot \Theta_1\mathfrak{u}_\rho^{(1)\flat}\sim \tau^{-1},\ \ \ (1,1,0,0)^T\cdot\Theta_2\mathfrak{u}_\rho^{(2)\flat}\sim \tau^{-1}.
    \end{align}
    Observe that
    \begin{align*}
         &\kappa_4:=(1,1,0,0)^T\cdot\mathfrak{u}_\rho^{(3)\flat}\\
         = &\Re \left(\left(
       \begin{array}{c}
        1 \\
        1 \\
       \end{array}
     \right)
     \cdot
    \left(\frac{e^{-i\rho\cdot x}}{i|\rho|} \left(
       \begin{array}{c}
         \partial_1(\nabla\cdot u^{(2)}) \\
         \partial_2(\nabla\cdot u^{(2)}) \\
       \end{array}
     \right)
     +\frac{e^{-i\tilde\rho\cdot x}}{i|\tilde\rho|} \left(
       \begin{array}{c}
         \partial_1(\nabla\cdot u^{(3)}) \\
         \partial_2(\nabla\cdot u^{(3)}) \\
       \end{array}
     \right)\right)\right).
   \end{align*}
   The $O(\tau)$ term of $ \kappa_4$ is
   $$
        \frac{1}{i\sqrt{2}\tau}(-\frac{\sqrt{\mu}}{\lambda+2\mu} (\rho^{(2)}_1+\rho^{(2)}_2)\rho^{(2)}\cdot r^{(2)}-\frac{\sqrt{\mu}}{\lambda+2\mu} (\rho^{(3)}_1+\rho^{(3)}_2)\rho^{(3)}\cdot r^{(3)})=0
   $$
   since $r^{(2)}=r^{(3)}$. Thus the leading order term of $ \kappa_4$ is $O(1)$, that is,

   $$
        \varphi=\varphi^{(2)}+\varphi^{(3)},
   $$
   where
   \begin{align*}
        \varphi^{(l)}=&\frac{1}{i\sqrt{2}\tau}\Big(\mu^{-1/2}(i\rho^{(l)}_2\partial_1 r^{(l)}_1+i\rho^{(l)}_1\partial_2 r^{(l)}_2+i\rho^{(l)}_1\partial_2 r^{(l)}_1+i\rho^{(l)}_2 \partial_1 r^{(l)}_2\\
        &+2i\rho\cdot \nabla r^{(l)})
        +\mu^{-1}(2i\rho\cdot (\nabla (\partial_1 s^{(l)})+\nabla(\partial_2 s^{(l)}))+i(\rho_1+\rho_2)\Delta s^{(l)})\\
        &+2i\partial_1\mu^{-1}\rho^{(l)}\cdot \nabla s^{(l)} -\partial_{11}\mu^{-1}i(\rho^{(l)}_1 +\rho^{(l)}_2) s^{(l)}\\
        &+2i\partial_2\mu^{-1}\rho^{(l)}\cdot \nabla s^{(l)}-\partial_{22}\mu^{-1}i(\rho^{(l)}_1+\rho^{(l)}_2 )s^{(l)}\\
        &+\partial_1\mu^{-1/2}(i\rho^{(l)}_2 r^{(l)}_1+i\rho^{(l)}_2 r^{(l)}_2+2i\rho^{(l)}_1r^{(l)}_1)\\
        &+\partial_2\mu^{-1/2}(i\rho^{(l)}_1r^{(l)}_1+i\rho^{(l)}_1r^{(l)}_2+2i\rho^{(l)}_2 r^{(l)}_2)
\Big).
   \end{align*}
Since
   \begin{align*}
       \kappa_1:=&(1,1,0,0)^T\cdot \mathfrak{u}_\rho^{(0)\flat}(x_0)\\
       &=\Re  \left(
       \begin{array}{c}
        1 \\
        1 \\
       \end{array}
     \right)
     \cdot
     \frac{-ie^{-i\rho\cdot x}}{|\rho|^2} \left(
       \begin{array}{c}
         \partial_1(\nabla\cdot u) \\
         \partial_2(\nabla\cdot u) \\
       \end{array}
     \right)(x_0)
     =\frac{-\sqrt{\mu}}{2(\lambda+2\mu)}+O(\tau^{-1})
   \end{align*}
   and (\ref{alpha}), we get that
   \begin{align}
         \kappa(x_0)
         =&(1,1,0,0)^T\cdot \mathrm{v}(x_0)\\
     = &\kappa_1(x_0)+\Theta_3 \kappa_4(x_0)+O(\tau^{-1})\\
     \sim &\frac{-\sqrt{\mu}}{2(\lambda+2\mu)}+\Theta_3(x_0)\varphi(x_0).
    \end{align}
Since $\mu\in \mathcal{P}$, $\varphi(x_0)$ is some fixed number. Moreover, we can take $s^{(0)}$ small such that $| \Theta_3(x_0)|$ to be sufficiently small (It can be done by following similar argument in Remark 2), then we can obtain that $\frac{-\sqrt{\mu}}{2(\lambda+2\mu)}+\Theta_3(x_0)\varphi(x_0)\neq 0.$ By continuity of $\kappa$, there exists a neighborhood $V$ such that $\kappa$ never vanishes in $V$.
\end{proof}

Let $\Omega_0=U\cap V\cap \Omega$. We have $\left\{\mathfrak{u}^{(1)\sharp}_\rho,\mathfrak{u}^{(2)\sharp}_\rho,\mathfrak{u}^{(3)\sharp}_\rho\right\}$are linearly independent in $\Omega_0$ as $\tau$ is sufficiently large and also $\kappa$ does not vanish in $\Omega_0$.
Then it follows that
\begin{align}\label{lambda123}
     \lambda= \frac{\sigma\mu}{\kappa}-\frac{k^2}{\kappa}\left(\mathfrak{u}_\rho^{(0)*}+\sum^3_{j=1}\Theta_j\mathfrak{u}_\rho^{(j)*}\right)\ \ \ \ \ \hbox{in $\Omega_0$}.
\end{align}

As in section \ref{localmu}, we also can find the solutions $u^{(j)}$ of the elasticity system such that
\begin{align}
    \|u^{(j)}-u^{(j)}_\rho\|_{C^{2}(\Omega)}
      <C \varepsilon,\ \ \ j=0,1,2,3.
\end{align}

Now we let the internal data $H_{x_0}$ contains the three solutions we constructed in Theorem \ref{mulocal} and the four solutions $u^{(j)},j=0,1,2,3$ in this section. Then we have the following result.

\begin{theorem}\label{lambdalocal}
Suppose that $(\lambda, \mu)$ and $(\tilde\lambda,\tilde\mu)\in \mathcal{P}$. For any fixed $x_0\in \partial\Omega$, let $u^{(j)}_\rho$ be the corresponding CGO solutions for $(\lambda,\mu)$ and $u^{(j)}$ constructed in sections 3.1, 3.2 with internal data $H_{x_0}$ and with $\varepsilon$ sufficiently small. Let $\tilde H_{x_0}$ be constructed similarly with the parameters $(\tilde\lambda,\tilde\mu)$. Assume that $\mu|_{\partial\Omega}=\tilde\mu|_{\partial\Omega}$.

Then $H_{x_0}=\tilde H_{x_0}$ implies that $ \lambda=\tilde\lambda$ in $\Omega_0$.
\end{theorem}
\begin{proof}
 Applying Theorem \ref{mulocal} and equation (\ref{lambda123}), we have the uniqueness of $\lambda$ near the point $x_0$.

\end{proof}

We deduce the following result by applying Lemma \ref{prelim2} and Theorem \ref{mu}.
\begin{theorem}\label{lambda}
Suppose that $(\lambda, \mu)$ and $(\tilde\lambda,\tilde\mu)\in \mathcal{P}$. For any fixed $x_0\in \overline\Omega$, let $u^{(j)}_\rho$ be the corresponding CGO solutions for $(\lambda,\mu)$ and $u^{(j)}$ constructed in sections 3.1, 3.2 with internal data $H_{x_0}$ and with $\varepsilon$ sufficiently small. Let $\tilde H_{x_0}$ be constructed similarly for the parameters $(\tilde\lambda,\tilde\mu)$ with $u^{(j)}|_{\partial\Omega}=\tilde u^{(j)}|_{\partial\Omega}$.
Then we have the estimates
\begin{align}\label{lambdaequation}
     \|\lambda-\tilde \lambda\|_{C(\Omega_0)}\leq C \left(|\mu(x_0)-\tilde \mu(x_0)|+\|H_\rho-\tilde H_\rho\|_{C^{2}(\Omega_0)}\right),\ \ \ x_0\in \partial\Omega.
\end{align}
and
\begin{align}\label{lambda2}
     \|\lambda-\tilde \lambda\|_{C(\Omega_0)}\leq C \left(|\mu(x^+_0)-\tilde \mu(x^+_0)|+\|H_\rho-\tilde H_\rho\|_{C^{2}(\Omega_0)}\right),\ \ \ x^+_0\in \partial\Omega_0,\ x_0\in \Omega.
\end{align}
\end{theorem}

With (\ref{lambda123}) and Theorem \ref{lambda}, we follow the same proof as in Theorem \ref{mu2}, then we can get the following global reconstruction of $\lambda$.

\begin{theorem}(Global reconstruction of $\lambda$)
Let $\Omega$ be an open bounded domain of $\mathbb{R}^2$ with smooth boundary. Suppose that the Lam\'{e} parameters $(\lambda, \mu)$ and $(\tilde\lambda,\tilde\mu)\in \mathcal{P}$ and $\mu|_{\partial\Omega}=\tilde\mu|_{\partial\Omega}$.
Let $u^{(j)}$ and $\tilde u^{(j)}$ be the solutions of the elasticity system with boundary data $g^{(j)}$ for parameters $(\lambda, \mu)$ and $(\tilde\lambda,\tilde\mu)$, respectively.
Let $H=(u^{(j)})_{1\leq j\leq J}$ and $\tilde H=(\tilde u^{(j)})_{1\leq j\leq J}$ be the corresponding internal data for $(\lambda, \mu)$ and $(\tilde\lambda,\tilde\mu)$, respectively for some integer $J\geq 7$ .

Then there is an open set of the boundary data $(g^{(j)})_{1\leq j\leq J}$ such that
if $H=\tilde H $
implies $\lambda=\tilde \lambda$ in $\Omega$.\\
Moreover, we have the stability estimate
$$
     \|\lambda-\tilde\lambda\|_{C(\Omega)}\leq C\|H-\tilde H \|_{C^2(\Omega)}.
$$
\end{theorem}

\section{Reconstruction of Lam\'{e} parameter in three-dimensional case}\label{global3}

The reconstruction of $\lambda$ and $\mu$ in $\mathbb{R}^3$ basically follows the similar argument as in Section \ref{global2}. In $\mathbb{R}^3$, we need more CGO solutions to get linearly independent vectors locally .

\subsection{Global reconstruction of $\mu$ in 3D}\label{3reconsmu}
Let $u=(u_1, u_2, u_3)^T$ be the displacement which satisfies the elasticity system
\begin{align}\label{3uu}
     \nabla\cdot (\lambda(\nabla\cdot u)I+2S(\nabla u)\mu)+k^2 u=0.
\end{align}
Denote
$$
     u^\sharp=\left(\begin{array}{c}
         \partial_1(\nabla\cdot u) \\
         \partial_2(\nabla\cdot u) \\
         \partial_3(\nabla\cdot u) \\
         \nabla\cdot u \\
         \nabla\cdot u \\
         \nabla\cdot u \\
       \end{array}
     \right),\ \
      F=\left(
       \begin{array}{c}
         \lambda+\mu \\
         \lambda+\mu \\
         \lambda+\mu \\
         \partial_1 (\lambda+\mu) \\
         \partial_2 (\lambda+\mu) \\
         \partial_3 (\lambda+\mu) \\
       \end{array}
     \right),\ \
     u^\flat=     \left(
       \begin{array}{c}
                b_{23} \\
                 b_{13} \\
                  b_{12} \\
         \partial_1( b_{23}) \\
         \partial_2( b_{13} ) \\
         \partial_3(b_{12})\\
       \end{array}
     \right),\ \
    G=     \left(
       \begin{array}{c}
        \partial_1\mu \\
        \partial_2\mu \\
        \partial_3\mu \\
         \mu \\
         \mu \\
         \mu \\
       \end{array}
     \right),
$$
where $b_{ij}=\partial_l u_l-\partial_i u_i-\partial_j u_j+\partial_i u_l+\partial_l u_i+\partial_j u_l+\partial_l u_j$ with  $l,i,j\in \{1,2,3\}$ are distinct numbers.
From (\ref{3uu}), we can deduce the following equation:
$$
     u^\sharp\cdot F+u^\flat\cdot G=-k^2u^*.
$$
Here $u^*=(u_1+u_2+u_3)$.

In the following we will show that how we can get four linearly independent vectors of the form $u^\sharp$
on some subdomain of $\Omega$.
The key thing is to observe the behavior of $u^\sharp$. We plug the CGO solutions $u_\rho=\mu^{-1/2}w_\rho+\mu^{-1}\nabla f_\rho-f_\rho\nabla \mu^{-1}$
into $u^\sharp$.
Then we get
\begin{align}\label{3v}
     u^\sharp_\rho=e^{i\rho\cdot x}\left(\left(
       \begin{array}{c}
         -\frac{\sqrt{\mu}}{\lambda+2\mu}\rho_1(\rho\cdot r)+O(|\rho|)\\
         -\frac{\sqrt{\mu}}{\lambda+2\mu}\rho_2(\rho\cdot r)+O(|\rho|)\\
         -\frac{\sqrt{\mu}}{\lambda+2\mu}\rho_3(\rho\cdot r)+O(|\rho|)\\
         i\frac{\sqrt{\mu}}{\lambda+2\mu}\rho\cdot r\\
         i\frac{\sqrt{\mu}}{\lambda+2\mu}\rho\cdot r\\
         i\frac{\sqrt{\mu}}{\lambda+2\mu}\rho\cdot r  \\
       \end{array}
     \right)
     +O(1)\right).
\end{align}
Note that $r=(r_1,r_2, r_3)^T$ and $\rho=(\rho_1,\rho_2, \rho_3)^T$.

Now we fix any point $x_0\in \overline\Omega$. Let $\rho=\tau(1,i,0)$ and $\tilde {\tilde\rho}=\tau(1, 0, i)$ with $\tau>0$. Let $\theta=\rho/\tau,\ \tilde{\tilde\theta}=\tilde{\tilde\rho}/\tau$. Since, in Lemma \ref{Eskin}, the matrix solutions $C_0(x,\theta)$ and $\tilde{\tilde C}_0(x,\tilde{\tilde\theta})$ are invertible, we can choose two constant vectors $g^{(0)}$ and $\tilde{\tilde g}^{(0)}$ such that $C_0(x,\theta)g^{(0)}=(r^{(0)},s^{(0)})^T$ and $\tilde{\tilde C}_0(x,\tilde{\tilde \theta})\tilde{\tilde g}^{(0)}=(\tilde{\tilde r}^{(0)},\tilde{\tilde s}^{(0)})^T$ with $s^{(0)}(x_0)=1=\tilde{\tilde s}^{(0)}(x_0)$ and $s^{(0)},\ \tilde{\tilde s}^{(0)}\neq 0$ and $\rho\cdot r^{(0)},\ \tilde{\tilde \rho}\cdot \tilde{\tilde r}^{(0)}\neq 0$ in a neighborhood of $x_0$, say $U_0$.
Then we have the CGO solutions of the elasticity system, that is,
\begin{align*}
u^{(0)}_\rho=&\ \mu^{-1/2}w^{(0)}_\rho+\mu^{-1}\nabla f^{(0)}_\rho-f^{(0)}_\rho\nabla\mu^{-1},\\
\tilde{\tilde u}^{(0)}_\rho=&\ \mu^{-1/2}\tilde{\tilde w}^{(0)}_\rho+\mu^{-1}\nabla \tilde{\tilde f}^{(0)}_\rho-\tilde{\tilde f}^{(0)}_\rho\nabla\mu^{-1}
\end{align*}
with
\begin{align*}
\left(\begin{array}{c}
   w^{(0)}_\rho \\
   f^{(0)}_\rho
\end{array}\right)
=&\ e^{i\rho\cdot x}\left(\left(
                         \begin{array}{c}
                           r^{(0)} \\
                           s^{(0)}\\
                         \end{array}
                       \right)
+O(\tau^{-1})\right),\\
\left(\begin{array}{c}
   \tilde{\tilde w}^{(0)}_\rho \\
   \tilde{\tilde f}^{(0)}_\rho
\end{array}\right)
=&\ e^{i\tilde{\tilde \rho}\cdot x}\left(\left(
                         \begin{array}{c}
                           \tilde{\tilde r}^{(0)} \\
                           \tilde{\tilde s}^{(0)}\\
                         \end{array}
                       \right)
+O(\tau^{-1})\right).
\end{align*}

Let $\tilde\rho=\tau(i,-1,0)$ and $\tilde\theta=\tilde\rho/\tau$. Let $C_1(x,\theta)$ and $C_2(x,\tilde\theta)$  satisfy that
$$
        i\theta\cdot \frac{\partial}{\partial x}C_1(x,\theta)=\theta \cdot V_1(x)C_1(x,\theta),\ \ \ i\tilde\theta\cdot \frac{\partial}{\partial x}C_2(x,\tilde\theta)=\tilde\theta \cdot V_1(x)C_2(x,\tilde\theta),
$$
respectively.
Since $\tilde\rho=i\rho$, we can choose $C_2(x,\tilde\theta)=C_1(x,\theta)$. Moreover, $r^{(2)}=r^{(1)}$ and $s^{(2)}=s^{(1)}$.
With suitable constant vector $g$, we can get that $s^{(l)}$ is zero at point $x_0$ and $r^{(1)}(x_0)=(1, -i, 0)=r^{(2)}(x_0)$. By continuity of $r^{(l)}$, we have $\rho\cdot r^{(l)}\neq 0$ in a neighborhood $U_1$ of $x_0$.  Then the CGO solution is
\begin{align*}
 u_\rho^{(l)}=\mu^{-1/2}w_\rho^{(l)}+\mu^{-1}\nabla f_\rho^{(l)}-f_\rho^{(l)}\nabla\mu^{-1}
\end{align*}
with
\begin{align*}
\left(\begin{array}{c}
   w_\rho^{(1)} \\
   f_\rho^{(1)}
\end{array}\right)
=&\ e^{i\rho\cdot x}\left(\left(
                         \begin{array}{c}
                           r^{(1)} \\
                           s^{(1)}\\
                         \end{array}
                       \right)
+O(\tau^{-1})\right), \\
 \left(\begin{array}{c}
   w_\rho^{(2)} \\
   f_\rho^{(2)}
\end{array}\right)
=&\ e^{i\tilde\rho\cdot x}\left(\left(
                         \begin{array}{c}
                           r^{(2)} \\
                           s^{(2)}\\
                         \end{array}
                       \right)
+O(\tau^{-1})\right).
\end{align*}

For $\tilde{\tilde \rho}$, with a suitable constant vector $\tilde g$, we can get that $s^{(3)}$ is zero at point $x_0$ and  $r^{(3)}(x_0)=(1, 0, -i)$. By continuity of $r^{(3)}$, we have $\tilde{\tilde \rho}\cdot r^{(3)}\neq 0$ in a neighborhood $U_2$ of $x_0$.  Then the CGO solution is
\begin{align*}
 u_\rho^{(3)}=\mu^{-1/2}w_\rho^{(3)}+\mu^{-1}\nabla f_\rho^{(3)}-f_\rho^{(3)}\nabla\mu^{-1}
\end{align*}
with
\begin{align*}
\left(\begin{array}{c}
   w_\rho^{(3)} \\
   f_\rho^{(3)}
\end{array}\right)
=e^{i\tilde{\tilde\rho} \cdot x}\left(\left(
                         \begin{array}{c}
                           r^{(3)} \\
                           s^{(3)}\\
                         \end{array}
                       \right)
+O(\tau^{-1})\right).
\end{align*}

Let $U=\cap_{l=0}^2 U_l$.
So far we have five CGO solutions, that is, $u_\rho^{(0)}, \tilde {\tilde u}_\rho^{(0)}, u_\rho^{(1)},  u_\rho^{(2)}$, and $u_\rho^{(3)}$.

Let $r^{(1)}=(r^{(1)}_1, r^{(1)}_2, r^{(1)}_3)$. We define
\begin{align*}
    \mathfrak{u}^1_\rho:=e^{-i\rho\cdot x}|\rho|^{-2}u^{(1)\sharp}_\rho=& \frac{1}{2}\left(\begin{array}{c}
         -\frac{\sqrt{\mu}}{\lambda+2\mu}(r^{(1)}_1+ir^{(1)}_2) \\
         -\frac{\sqrt{\mu}}{\lambda+2\mu}(ir^{(1)}_1-r^{(1)}_2)\\
         0\\
         0\\
         0 \\
         0 \\
       \end{array}
       \right)
     +O(\tau^{-1})
\end{align*}
and
\begin{align*}
    \mathfrak{u}^{2,1}_\rho:=e^{-i\rho\cdot x}|\rho|^{-1}u^{(1)\sharp}_\rho=& \frac{1}{\sqrt{2}}\left(\begin{array}{c}
         -\frac{\sqrt{\mu}}{\lambda+2\mu}\tau(r^{(1)}_1+ir^{(1)}_2)+O(1)\\
         -\frac{\sqrt{\mu}}{\lambda+2\mu}\tau(ir^{(1)}_1-r^{(1)}_2)+O(1)\\
         0\\
         i\frac{\sqrt{\mu}}{\lambda+2\mu}( r^{(1)}_1+ir^{(1)}_2) \\
         i\frac{\sqrt{\mu}}{\lambda+2\mu}( r^{(1)}_1+ir^{(1)}_2) \\
         i\frac{\sqrt{\mu}}{\lambda+2\mu}( r^{(1)}_1+ir^{(1)}_2) \\
       \end{array}
       \right)
     +O(\tau^{-1})
\end{align*}
and
\begin{align*}
    \mathfrak{u}^{2,2}_\rho:=e^{-i\tilde\rho\cdot x}|\tilde\rho|^{-1}u^{(2)\sharp}_\rho=& \frac{1}{\sqrt{2}}\left(\begin{array}{c}
         -\frac{\sqrt{\mu}}{\lambda+2\mu}\tau(-r^{(2)}_1-ir^{(2)}_2)+O(1) \\
         -\frac{\sqrt{\mu}}{\lambda+2\mu}\tau(-ir^{(2)}_1+r^{(2)}_2)+O(1) \\
         0\\
         i\frac{\sqrt{\mu}}{\lambda+2\mu}( ir^{(2)}_1-r^{(2)}_2) \\
         i\frac{\sqrt{\mu}}{\lambda+2\mu}( ir^{(2)}_1-r^{(2)}_2) \\
         i\frac{\sqrt{\mu}}{\lambda+2\mu}( ir^{(2)}_1-r^{(2)}_2) \\
       \end{array}
       \right)
     +O(\tau^{-1}).
\end{align*}
Let $\mathfrak{u}^2_\rho=\mathfrak{u}^{2,1}_\rho+\mathfrak{u}^{2,2}_\rho$. Since $r^{(2)}=r^{(1)}$ and $s^{(2)}=s^{(1)}$, we have
\begin{align*}
    \mathfrak{u}^2_\rho= \frac{1}{\sqrt{2}}\left(\begin{array}{c}
         O(1) \\
         O(1) \\
         0\\
         i\frac{\sqrt{\mu}}{\lambda+2\mu}( (1+i)r^{(1)}_1+(-1+i)r^{(1)}_2) \\
         i\frac{\sqrt{\mu}}{\lambda+2\mu}( (1+i)r^{(1)}_1+(-1+i)r^{(1)}_2) \\
         i\frac{\sqrt{\mu}}{\lambda+2\mu}((1+i)r^{(1)}_1+(-1+i)r^{(1)}_2) \\
       \end{array}
       \right)
     +O(\tau^{-1}).
\end{align*}
We also define
\begin{align*}
    \mathfrak{u}^3_\rho:=e^{-i\tilde{\tilde\rho}\cdot x}|\rho|^{-2}u^{(3)\sharp}_\rho=& \frac{1}{2}\left(\begin{array}{c}
         -\frac{\sqrt{\mu}}{\lambda+2\mu}(r^{(3)}_1+ir^{(3)}_3) \\
         0\\
         -\frac{\sqrt{\mu}}{\lambda+2\mu}(ir^{(3)}_1-r^{(3)}_3)\\
         0\\
         0 \\
         0 \\
       \end{array}
       \right)
     +O(\tau^{-1})
\end{align*}

We denote
\begin{align*}
        \mathfrak{u}^{(0)\natural}_{1\rho}=\Re{e^{-i\rho\cdot x}|\rho|^{-2} u^{(0)\natural}_\rho};\ \ \mathfrak{u}^{(0)\natural}_{2\rho}=\Im{e^{-i\rho\cdot x}|\rho|^{-2} u^{(0)\natural}_\rho};\\
        \mathfrak{u}^{(0)\natural}_{3\rho}=\Re{e^{-i\tilde{\tilde\rho}\cdot x}|\rho|^{-2} \tilde{\tilde u}^{(0)\natural}_\rho};\ \ \ \natural=\sharp,\flat,*
\end{align*}
and
\begin{align*}
\mathfrak{u}^{(1)\natural}_\rho=\Re e^{-i\rho\cdot x}|\rho|^{-2}u^{(1)\natural}_\rho ;\ \
\mathfrak{u}^{(2)\natural}_\rho=\Im e^{-i\rho\cdot x}|\rho|^{-2}u^{(1)\natural}_\rho;\\
\mathfrak{u}^{(3)\natural}_\rho=\Re\left(e^{-i\rho\cdot x}|\rho|^{-1}u^{(1)\natural}_\rho+e^{-i\tilde\rho\cdot x}|\tilde\rho|^{-1}u^{(2)\natural}_\rho\right);\ \
\mathfrak{u}^{(4)\natural}_\rho=\Re e^{-i\tilde{\tilde\rho}\cdot x}|\rho|^{-2}u^{(3)\natural}_\rho.
\end{align*}

Then $\left\{\mathfrak{u}^{(j)\sharp}_\rho:\  1\leq j\leq 4\right\}$ are linearly independent in the neighborhood $U$ of $x_0$ as $\tau$ is sufficiently large. Therefore, for fixed $l=1,2,3$ there exist functions $\Theta^l_j,\  j=1,2,3,4,$ such that
$$
\mathfrak{u}^{(0)\sharp}_{l\rho} +\sum^4_{j=1}\Theta_j^l\mathfrak{u}^{(j)\sharp}_{\rho} =0.
$$

As in Section \ref{localmu}, we summing over equations, then
we have
$$
     \beta_{\rho,l}\cdot \nabla\mu
     +
     \gamma_{\rho,l} \mu
=-k^2\left(\mathfrak{u}^{(0)*}_{l\rho} +\sum^4_{j=1}\Theta_j^l\mathfrak{u}^{(j)*}_{\rho}\right) \ \ \ \hbox{ for $l=1, 2, 3$,}
$$
where
$\beta_{\rho,l}$ and $\gamma_{\rho,l}$ are functions which depend on $\rho, \Omega$ and CGO solutions $\tilde {\tilde u}_\rho^{(0)},u_\rho^{(j)}$ for $j=0,\dots, 3$.

\begin{lemma}\label{3beta11}
Given any point $x_0\in \overline\Omega$, there exists an open neighborhood $V$ of $x_{0}$ such that $\beta_{\rho,l}$ is not zero in $V$, for $l=1, 2, 3$.
\end{lemma}

\begin{proof}
Following the similar proof as in Lemma \ref{beta11}, we can prove this Lemma.

\end{proof}

Let $\Omega_0=U\cap V\cap \Omega$. Based on the lemma above, we may suppose that $\beta_{\rho,j}$ are linearly independent in $\Omega_0$ as $\tau$ sufficiently large.

Let $g^{(j)}_\rho=u_\rho^{(j)}|_{\partial\Omega}$ and $\tilde{\tilde g}^{(0)}_\rho=\tilde {\tilde u}_\rho^{(0)}|_{\partial\Omega}$ be the given boundary data for $j=0,1,2,3$. Let $\tilde{\tilde g}^{(0)},\ g^{(j)}\in C^{1,\alpha}(\partial\Omega)$ be the boundary data close to $\tilde{\tilde g}^{(0)}_\rho,\ g^{(j)}_\rho $, respectively, that is,
\begin{align*}
\|g^{(j)}-g^{(j)}_\rho\|_{C^{1,\alpha}(\partial\Omega)}< \varepsilon,\ \ \ \|\tilde{\tilde g}^{(0)}-\tilde {\tilde g}^{(0)}_\rho\|_{C^{1,\alpha}(\partial\Omega)}< \varepsilon.
\end{align*}
Then we can find solutions $\tilde{\tilde u}^{(0)},\ u^{(j)}$ of the elasticity system with boundary data $\tilde{\tilde g}^{(0)},g^{(j)}$, respectively. By regularity theorem, it follows that
\begin{align}\label{3uuu}
    \|u^{(j)}-u^{(j)}_\rho\|_{C^{2}(\Omega)}
      <C \varepsilon,\ \ \ \ \ \|\tilde{\tilde u}^{(0)}-\tilde {\tilde u}^{(0)}_\rho\|_{C^{2}(\Omega)}
      <C \varepsilon.
\end{align}
From (\ref{3uuu}), we have
\begin{align*}
     \|\mathfrak{u}^{(j)\sharp}-\mathfrak{u}^{(j)\sharp}_\rho\|_{C(\Omega)}\leq C\varepsilon,\ \ 1\leq j\leq 4.
\end{align*}
Therefore, $\left\{\mathfrak{u}^{(j)\sharp}:\  1\leq j\leq 4\right\}$ are also linearly independent when $\varepsilon$ is sufficiently small.
We construct $\beta_l$ by replacing $u^{(j)}_\rho$ by $u^{(j)}$. Then from (\ref{3uuu}), we can deduce that
$$
     \|\beta_l-\beta_{\rho,l}\|_{C^1(\Omega)}
$$
is small when $\varepsilon$ is sufficiently small.
Then we have the following equation
\begin{align}\label{3muj}
     \beta_l\cdot\nabla\mu+\gamma_l\mu=-k^2\left(\mathfrak{u}^{(0)*}_{l} +\sum^4_{j=1}\Theta_j^l\mathfrak{u}^{(j)*}\right),\ \ \ \ l=1,2,3,
\end{align}
with $\{\beta_l(x)\}_{l=1,2,3}$ a basis in $\mathbb{R}^3$ for every point $x\in \Omega_0$.
There exists an invertible matrix $A=(a_{ij})$ such that $\beta_l=\sum a_{lk}e_k$ with inverse of class $C(\Omega)$. Thus, we have constructed two vector-valued functions $\Gamma(x),\ \Phi(x)\in C(\Omega)$ such that (\ref{3muj}) can be rewritten as
\begin{align}
     \nabla\mu+\Gamma(x)\mu=\Phi(x)\ \ \ \ \ \hbox{in\ $\Omega_0$}.
\end{align}

Then we have the following uniqueness and stability theorem.

\begin{theorem}(Global reconstruction of $\mu$)\label{mu3}
Let $\Omega$ be an open bounded domain of $\mathbb{R}^3$ with smooth boundary. Suppose that the Lam\'{e} parameters $(\lambda, \mu)$ and $(\tilde\lambda,\tilde\mu)\in \mathcal{P}$ and $\mu|_{\partial\Omega}=\tilde\mu|_{\partial\Omega}$.
Let $u^{(j)}$ and $\tilde u^{(j)}$ be the solutions of the elasticity system with boundary data $g^{(j)}$ for parameters $(\lambda, \mu)$ and $(\tilde\lambda,\tilde\mu)$, respectively.
Let $H=(u^{(j)})_{1\leq j\leq J}$ and $\tilde H=(\tilde u^{(j)})_{1\leq j\leq J}$ be the corresponding internal data for $(\lambda, \mu)$ and $(\tilde\lambda,\tilde\mu)$, respectively for some integer $J\geq 5$ .

Then there is an open set of the boundary data $(g^{(j)})_{1\leq j\leq J}$ such that
if $H=\tilde H $
implies $\mu=\tilde \mu$ in $\Omega$.\\
Moreover, we have the stability estimate
$$
     \|\mu-\tilde\mu\|_{C(\Omega)}\leq C\|H-\tilde H \|_{C^2(\Omega)}.
$$
\end{theorem}
\begin{proof}
  The proof is similar to Theorem \ref{mu2}.
\end{proof}

\subsection{Global reconstruction of $\lambda$ in 3D}\label{3reconslambda}
The elasticity system can also be written in this form
    \begin{align}\label{3equ}
     u^\sharp\cdot F
     +
    u^\flat\cdot G
     =-k^2u^*,
\end{align}
where
$$
   u^\sharp= \left(\begin{array}{c}
         \nabla\cdot u \\
         \nabla\cdot u \\
         \nabla\cdot u \\
           b_{23} \\
           b_{13} \\
           b_{12} \\
       \end{array}
     \right),\ \
     F=\left(
       \begin{array}{c}
         \partial_1 (\lambda+\mu) \\
         \partial_2 (\lambda+\mu) \\
         \partial_3 (\lambda+\mu) \\
         \partial_1\mu \\
        \partial_2\mu \\
        \partial_3\mu \\
       \end{array}
     \right),\ \
     u^\flat=    \left(
       \begin{array}{c}
          \partial_1(\nabla\cdot u) \\
         \partial_2(\nabla\cdot u) \\
         \partial_3(\nabla\cdot u) \\
         \partial_1( b_{23}) \\
         \partial_2( b_{13} ) \\
         \partial_3(b_{12})\\
       \end{array}
     \right)
,\ \
  G=     \left(
       \begin{array}{c}
         \lambda+\mu \\
         \lambda+\mu \\
         \lambda+\mu \\
         \mu \\
         \mu \\
         \mu \\
       \end{array}
     \right).
$$
As in the reconstruction of $\mu$, we will construct CGO solutions such that the first term of (\ref{3equ}) can be eliminated.

Plug the CGO solution $ u_\rho=\mu^{-1/2}w_\rho+\mu^{-1}\nabla f_\rho-f_\rho\nabla\mu^{-1}$ into $u^\sharp$.
Then we get
\begin{align*}
     u^\sharp_\rho=e^{i\rho\cdot x}\left(\left(
       \begin{array}{c}
         i\frac{\sqrt{\mu}}{\lambda+2\mu}\rho\cdot r \\
         i\frac{\sqrt{\mu}}{\lambda+2\mu}\rho\cdot r \\
         i\frac{\sqrt{\mu}}{\lambda+2\mu}\rho\cdot r \\
         -2\mu^{-1}\rho_1(\rho_1+\rho_2+\rho_3)s+O(|\rho|) \\
         -2\mu^{-1}\rho_2(\rho_1+\rho_2+\rho_3)s+O(|\rho|) \\
         -2\mu^{-1}\rho_3(\rho_1+\rho_2+\rho_3)s+O(|\rho|) \\
        \end{array}
     \right)
     +O(1)\right).
\end{align*}

For the same fixed point $x_0\in \overline\Omega$.
We choose a constant vector $g^{(0)}$ such that $C_0(x,\theta)g^{(0)}=(r^{(0)}, s^{(0)})$ with $s^{(0)}\neq 0$ and $\rho\cdot r^{(0)}\neq 0$ in a neighborhood of $x_0$, say $U_0$, and $\rho\cdot r^{(0)}(x_0)=1$. Then we get the CGO solution of the elasticity system
$$
  u_\rho^{(0)}=\mu^{-1/2}w_\rho^{(0)}+\mu^{-1}\nabla f_\rho^{(0)}-f_\rho^{(0)}\nabla \mu^{-1}
$$
with
\begin{align*}
\left(\begin{array}{c}
   w_\rho^{(0)} \\
   f_\rho^{(0)}
\end{array}\right)
=e^{i\rho\cdot x}\left(\left(
                         \begin{array}{c}
                           r^{(0)} \\
                           s^{(0)}\\
                         \end{array}
                       \right)
+O(\tau^{-1})\right).
\end{align*}

We choose another constant vector $g^{(1)}$ such that $C_1(x,\theta)g^{(1)}=(r^{(1)}, s^{(1)})$ with $s^{(1)}\neq 0$ in a neighborhood of $x_0$, say $U_1$, and $\rho\cdot r^{(1)}(x_0)=0$. Then we get the CGO solution of the elasticity system
$$
  u_\rho^{(1)}=\mu^{-1/2}w_\rho^{(1)}+\mu^{-1}\nabla f_\rho^{(1)}-f_\rho^{(1)}\nabla \mu^{-1}
$$
with
\begin{align*}
\left(\begin{array}{c}
   w_\rho^{(1)} \\
   f_\rho^{(1)}
\end{array}\right)
=e^{i\rho\cdot x}\left(\left(
                         \begin{array}{c}
                           r^{(1)} \\
                           s^{(1)}\\
                         \end{array}
                       \right)
+O(\tau^{-1})\right).
\end{align*}

For $l=2,3$, we choose a constant vector $g^{(l)}$ such that $C_l(x,\theta)g^{(l)}=(r^{(l)}, s^{(l)})$ with $\rho\cdot r^{(l)}\neq 0$ in a neighborhood of $x_0$, say $U_2$. Here we can choose $r^{(2)}=r^{(3)},\ s^{(2)}=s^{(3)}$ by taking $g^{(2)}=g^{(3)}$ and $C_2(x,\theta)=C_3(x,\tilde\theta)$. Then we get the CGO solution of the elasticity system
$$
  u_\rho^{(l)}=\mu^{-1/2}w_\rho^{(l)}+\mu^{-1}\nabla f_\rho^{(l)}-f_\rho^{(l)}\nabla \mu^{-1}
$$
with
\begin{align*}
\left(\begin{array}{c}
   w_\rho^{(2)} \\
   f_\rho^{(2)}
\end{array}\right)
=&\ e^{i\rho\cdot x}\left(\left(
                         \begin{array}{c}
                           r^{(2)} \\
                           s^{(2)}\\
                         \end{array}
                       \right)
+O(\tau^{-1})\right),\\
 \left(\begin{array}{c}
   w_\rho^{(3)} \\
   f_\rho^{(3)}
\end{array}\right)
=&\ e^{i\tilde\rho\cdot x}\left(\left(
                         \begin{array}{c}
                           r^{(3)} \\
                           s^{(3)}\\
                         \end{array}
                       \right)
+O(\tau^{-1})\right).
\end{align*}
For $\tilde{\tilde\rho}$, we choose another constant vector $g^{(4)}$ such that $C_4(x,\theta)g^{(4)}=(r^{(4)}, s^{(4)})$ with $s^{(4)}\neq 0$ in a neighborhood of $x_0$, say $U_3$, and $\tilde{\tilde\rho}\cdot r^{(4)}(x_0)=0$. Then we get the CGO solution of the elasticity system
$$
  u_\rho^{(4)}=\mu^{-1/2}w_\rho^{(4)}+\mu^{-1}\nabla f_\rho^{(4)}-f_\rho^{(4)}\nabla \mu^{-1}
$$
with
\begin{align*}
\left(\begin{array}{c}
   w_\rho^{(4)} \\
   f_\rho^{(4)}
\end{array}\right)
=e^{i\tilde{\tilde\rho}\cdot x}\left(\left(
                         \begin{array}{c}
                           r^{(4)} \\
                           s^{(4)}\\
                         \end{array}
                       \right)
+O(\tau^{-1})\right).
\end{align*}

We define
$$
     \mathfrak{u}_\rho^1:=e^{-i\rho\cdot x}|\rho|^{-2}u^{(1)\sharp}_\rho=-2\mu^{-1}s
     \left(
       \begin{array}{c}
         0 \\
         0 \\
         0 \\
         1+i \\
         i-1 \\
         0 \\
       \end{array}
     \right)
     +O(\tau^{-1})
$$
and
\begin{align*}
      \mathfrak{u}_\rho^{2,1}
      :=&e^{-i\rho\cdot x}|\rho|^{-1}u^{(2)\sharp}_\rho\\
      =&\frac{1}{\sqrt{2}}\left(
       \begin{array}{c}
         i\frac{\sqrt{\mu}}{\lambda+2\mu}\rho\cdot r^{(2)} \\
         i\frac{\sqrt{\mu}}{\lambda+2\mu}\rho\cdot r^{(2)}\\
         i\frac{\sqrt{\mu}}{\lambda+2\mu}\rho\cdot r^{(2)}\\
         -2\mu^{-1}\tau(1+i)s^{(2)}+O(1) \\
         -2\mu^{-1}\tau(i-1)s^{(2)}+O(1) \\
         O(1)  \\
       \end{array}
     \right)
     +O(\tau^{-1})
\end{align*}
and
\begin{align*}
      \mathfrak{u}_\rho^{2,2}
      :=&e^{-i\tilde\rho\cdot x}|\tilde\rho|^{-1}u^{(3)\sharp}_\rho\\
      =&\frac{1}{\sqrt{2}}\left(
       \begin{array}{c}
         i\frac{\sqrt{\mu}}{\lambda+2\mu}\tilde\rho\cdot r^{(3)} \\
         i\frac{\sqrt{\mu}}{\lambda+2\mu}\tilde\rho\cdot r^{(3)}\\
         i\frac{\sqrt{\mu}}{\lambda+2\mu}\tilde\rho\cdot r^{(3)}\\
         -2\mu^{-1}\tau(-1-i)s^{(3)}+ O(1) \\
         -2\mu^{-1}\tau(1-i)s^{(3)}+ O(1) \\
         O(1) \\
       \end{array}
     \right)
     +O(\tau^{-1}).
\end{align*}
Let $\mathfrak{u}_\rho^2=\mathfrak{u}_\rho^{2,1}+\mathfrak{u}_\rho^{2,2}$, then the higher order is eliminated. Thus we have
\begin{align*}
    \mathfrak{u}_\rho^2=\frac{1}{\sqrt{2}}\left(
       \begin{array}{c}
         i\frac{\sqrt{\mu}}{\lambda+2\mu}(\rho+\tilde\rho)\cdot r^{(3)} \\
         i\frac{\sqrt{\mu}}{\lambda+2\mu}(\rho+\tilde\rho)\cdot r^{(3)} \\
         i\frac{\sqrt{\mu}}{\lambda+2\mu}(\rho+\tilde\rho)\cdot r^{(3)}\\
         O(1) \\
         O(1) \\
         O(1) \\
       \end{array}
     \right)
     +O(\tau^{-1}).
\end{align*}
Also, we define that
$$
     \mathfrak{u}_\rho^3:=e^{-i\rho\cdot x}|\rho|^{-2}u^{(4)\sharp}_\rho=-2\mu^{-1}s
     \left(
       \begin{array}{c}
         0 \\
         0 \\
         0 \\
         1+i \\
         0 \\
         i-1 \\
       \end{array}
     \right)
     +O(\tau^{-1}).
$$

We denote
\begin{align*}
        \mathfrak{u}^{(0)\natural}_{\rho}=\Re{e^{-i\rho\cdot x}|\rho|^{-2} u^{(0)\natural}_\rho};\ \ \ \natural=\sharp,\flat,*
\end{align*}
and
\begin{align*}
\mathfrak{u}^{(1)\natural}_\rho=\Re e^{-i\rho\cdot x}|\rho|^{-2}u^{(1)\natural}_\rho ;\ \
\mathfrak{u}^{(2)\natural}_\rho=\Im e^{-i\rho\cdot x}|\rho|^{-2}u^{(1)\natural}_\rho;\\
\mathfrak{u}^{(3)\natural}_\rho=\Re\left(e^{-i\rho\cdot x}|\rho|^{-1}u^{(2)\natural}_\rho+e^{-i\tilde\rho\cdot x}|\tilde\rho|^{-1}u^{(3)\natural}_\rho\right);\ \
\mathfrak{u}^{(4)\natural}_\rho=\Re e^{-i\rho\cdot x}|\rho|^{-2}u^{(4)\natural}_\rho.
\end{align*}

Then $\left\{\mathfrak{u}^{(j)\sharp}_\rho:\  1\leq j\leq 4\right\}$ are linearly independent in the neighborhood $U=\cap_{l=0}^3 U_l$ of $x_0$ as $\tau$ is sufficiently large. Therefore, there exist functions $\Theta_j,\  j=1,2,3,4,$ such that
$$
\mathfrak{u}^{(0)\sharp}_{\rho} +\sum^4_{j=1}\Theta_j\mathfrak{u}^{(j)\sharp}_{\rho} =0.
$$
Summing over, we get the following equation as in Section \ref{reconslambda}:
\begin{align}\label{3frac}
     \mathrm{v}
     \cdot G
     =-k^2\left(\mathfrak{u}^{(0)*}_{\rho} +\sum^4_{j=1}\Theta_j\mathfrak{u}^{(j)*}_{\rho}\right),
\end{align}
with $\mathrm{v}=\mathfrak{u}^{(0)\flat}_{\rho} +\sum^4_{j=1}\Theta_j\mathfrak{u}^{(j)\flat}_{\rho}$.
We obtain that
\begin{align}\label{3lambdamu}
\kappa\lambda=\sigma\mu-k^2\left(\mathfrak{u}^{(0)*}_{\rho} +\sum^4_{j=1}\Theta_j\mathfrak{u}^{(j)*}_{\rho}\right),
\end{align}
where
\begin{align*}
     \kappa= (1,1,1,0,0,0)^T\cdot  \mathrm{v},\\
     \sigma=-(1,1,1,1,1,1)^T\cdot    \mathrm{v}.
\end{align*}
\begin{lemma}
     $\kappa$ does not vanish in some neighborhood of $x_0$.
\end{lemma}
\begin{proof}
    Similar argument as Lemma \ref{kappa}.
   Since
   $$
   \left(
       \begin{array}{c}
        \partial_1(\nabla\cdot u_\rho) \\
         \partial_2(\nabla\cdot u_\rho) \\
         \partial_3(\nabla\cdot u_\rho) \\
       \end{array}
     \right)=e^{i\rho\cdot x}\left(\left(
       \begin{array}{c}
         -\frac{\sqrt{\mu}}{\lambda+2\mu}\rho_1(\rho\cdot r)+O(|\rho|)\\
         -\frac{\sqrt{\mu}}{\lambda+2\mu}\rho_2(\rho\cdot r)+O(|\rho|)\\
         -\frac{\sqrt{\mu}}{\lambda+2\mu}\rho_2(\rho\cdot r)+O(|\rho|)\\
       \end{array}
     \right)
     +O(1)\right)
   $$
    and $\rho\cdot r^{(1)}(x_0)=0=\tilde{\tilde\rho}\cdot r^{(4)}(x_0)$, we have
    \begin{align*}
    (1,1,1,0,0,0)^T\cdot \left(\Theta_1\mathfrak{u}^{(1)\flat}_{\rho} +\Theta_2\mathfrak{u}^{(2)\flat}_{\rho} +\Theta_4\mathfrak{u}^{(4)\flat}_{\rho} \right)(x_0)
     \sim \tau^{-1}.
    \end{align*}
    Hence, we obtain that
      \begin{align*}
      \kappa(x_0)&=(1,1,1,0,0,0)^T\cdot \left(\mathfrak{u}^{(0)\flat}_{\rho} +\Theta_3\mathfrak{u}^{(3)\flat}_{\rho}\right)(x_0) +O(\tau^{-1})\\
     &\sim -\frac{\sqrt{\mu}}{2(\lambda+2\mu)}+
    \Theta_3\mathfrak{u}^{(3)\flat}_{\rho} (x_0).
   \end{align*}
     We can take $s^{(0)}$ small enough such that $\Theta_3(x_0)$ is small. Thus, $\kappa(x_0)\neq 0$. By continuity of $\kappa$, $\kappa$ does not vanish in some neighborhood $V$ of $x_0$.
\end{proof}
Let $\Omega_0=U\cap V\cap \Omega$. Since $\kappa$ does not vanish in $\Omega_0$,
we have
\begin{align}\label{3ll}
     \lambda= \frac{\sigma\mu}{\kappa}-\frac{k^2}{\kappa}\left(\mathfrak{u}_\rho^{(0)*}+\sum^4_{j=1}\Theta_j\mathfrak{u}_\rho^{(j)*}\right)\ \ \ \ \ \hbox{in $\Omega_0$}.
\end{align}
Applying the similar proof as in Theorem \ref{mu2}, we can deduce the following result.

\begin{theorem}(Global reconstruction of $\lambda$)
Let $\Omega$ be an open bounded domain of $\mathbb{R}^3$ with smooth boundary. Suppose that the Lam\'{e} parameters $(\lambda, \mu)$ and $(\tilde\lambda,\tilde\mu)\in \mathcal{P}$ and $\mu|_{\partial\Omega}=\tilde\mu|_{\partial\Omega}$.
Let $u^{(j)}$ and $\tilde u^{(j)}$ be the solutions of the elasticity system with boundary data $g^{(j)}$ for parameters $(\lambda, \mu)$ and $(\tilde\lambda,\tilde\mu)$, respectively.
Let $H=(u^{(j)})_{1\leq j\leq J}$ and $\tilde H=(\tilde u^{(j)})_{1\leq j\leq J}$ be the corresponding internal data for $(\lambda, \mu)$ and $(\tilde\lambda,\tilde\mu)$, respectively for some integer $J\geq 10$ .

Then there is an open set of the boundary data $(g^{(j)})_{1\leq j\leq J}$ such that
if $H=\tilde H $
implies $\lambda=\tilde \lambda$ in $\Omega$.\\
Moreover, we have the stability estimate
$$
     \|\lambda-\tilde\lambda\|_{C(\Omega)}\leq C\|H-\tilde H \|_{C^2(\Omega)}.
$$
\end{theorem}

\textbf{Acknowlegments}\\

The author are grateful to professor Gunther Uhlmann for his encouragements and helpful discussions. The author also would like to thank Professor Jenn-Nan Wang for taking the time to discuss some properties of the elasticity system with her. The author is partially supported by NSF.

\end{document}